\documentclass[a4paper,11pt]{article}

\usepackage{fullpage}
\usepackage{amsthm,amsmath, amssymb}
\usepackage{palatino,comment}
\usepackage[usenames, dvipsnames]{color}
\usepackage[active]{srcltx}   
\input{xy}
\xyoption{all}
\xyoption{matrix}

\usepackage{soul}

\newtheorem{theorem}{Theorem}[section]
\newtheorem{teo}[theorem]{Theorem}

\newtheorem{lem}[theorem]{Lemma}

\newtheorem{prop}[theorem]{Proposition}
\newtheorem{corollary}[theorem]{Corollary}

\newtheorem{coro}[theorem]{Corollary}

\theoremstyle{definition}
\newtheorem{definition}[theorem]{Definition}

\newtheorem{defi}[theorem]{Definition}

\newtheorem{notation}[theorem]{Notation}

\newtheorem{remark}[theorem]{Remark}

\newtheorem{ex}[theorem]{Example}

\newtheorem{rem}[theorem]{Remark}
\newtheorem{question}[theorem]{Question}

\DeclareMathOperator{\Cof}{Cof}
\DeclareMathOperator{\Ker}{Ker}

\DeclareMathOperator{\id}{id}

\DeclareMathOperator{\End}{End}
\DeclareMathOperator{\Fun}{Fun}
\DeclareMathOperator{\GL}{GL} 
\DeclareMathOperator{\SL}{SL} 
\DeclareMathOperator{\Sp}{Sp}
\DeclareMathOperator{\SO}{SO}
\DeclareMathOperator{\Or}{O}
\DeclareMathOperator{\Mq}{M}

\DeclareMathOperator{\Aut}{Aut}
\DeclareMathOperator{\ev}{ev}
\DeclareMathOperator{\coev}{coev}

\newcommand{\ot}{\otimes}

\newcommand{\wt}{\widetilde}
\newcommand{\wh}{\widehat}

\newcommand{\bq}{\textbf{q}}

\def\M{\mathcal M}

\def\SS{\mathcal S}

\def\fR{\mathfrak{R}}

\def\b{\mathfrak b}

\def\B{\mathfrak B}
\def\J{\mathfrak J}

\def\Z{\mathbb Z}
\def\s{\mathbb S}
\def\N{\mathbb N}
\def\O{\mathcal O}

\def\om{\omega}

\def\t{\triangleleft}
\def\trid{\triangleright}
\def\mO{\mathcal{O}}

\def\bt{\mathbf{t}}

\def\eps{\varepsilon}
\def\com{\Delta}

\def\Im{\mathrm{Im}}

\def\pf{\begin{proof}}
\def\epf{\end{proof}}

\begin{document}

\author{Marco Andr\'es Farinati
\thanks{Partially supported by CONICET, UBACyT}
\and 
Gast{\'o}n Andr{\'e}s Garc{\'\i}a\thanks{Partially supported by CONICET, ANPCyT, Secyt.\newline
\noindent 2010 \emph{Mathematics Subject Classification:}\,  17B37, 20G42, 16W20.  \newline
   \emph{Keywords:} Quantum function algebras, Nichols algebras, quantum determinants.  
}}

\title{Quantum function algebras from finite-dimensional Nichols algebras}

\date{}

\newcommand{\Addresses}{{
  \bigskip
  \footnotesize
  \noindent \textsc{M.~A.~Farinati
  \newline I.M.A.S. CONICET - Departamento de Matem\'atica,
  \newline  F.C.E.y N., 
Universidad de Buenos Aires 
   \newline
Ciudad Universitaria Pabell\'on I \newline
(1428) Ciudad de Buenos Aires, Argentina}
\par\nopagebreak
\noindent  \textit{E-mail address:} \texttt{mfarinat@dm.uba.ar}

  \medskip

 \noindent \textsc{G.~A.~Garc\'ia 
  \newline CMaLP, Departamento de Matem\'atica, 
   \newline Facultad de Ciencias Exactas
\newline
 Universidad Nacional de La Plata   ---   CONICET
   \newline
 C. C. 172   \, --- \,   1900 La Plata, Argentina}
 \par\nopagebreak
 \noindent  \textit{E-mail address:} \texttt{ggarcia@mate.unlp.edu.ar}
}}

\maketitle

\begin{abstract}
We describe how to find quantum determinants and antipode formulas
 from braided vector spaces 
 using the FRT-construction and finite-dimensional Nichols algebras. It improves the 
 construction of quantum function algebras using quantum grassmanian algebras. 
Given a finite-dimensional Nichols algebra $\B$, our method 
provides a Hopf algebra  $H$ such that $\B$ is a braided Hopf 
algebra in the category of $H$-comodules.
It also serves as source to produce Hopf algebras generated by
 cosemisimple subcoalgebras, which are of interest for the 
generalized lifting method.
 We give several examples, among them quantum function algebras from
 Fomin-Kirillov algebras
 associated with the symmetric group on three letters.
\end{abstract}


\section*{Introduction}
Let $\Bbbk$ be a field and $V$  a finite-dimensional $\Bbbk$-vector space. 
A map $c\in \Aut(V\ot V)$ is called a braiding if it satisfies  
the braid equation
\begin{equation}\label{eq:braideq}
(c\ot \id)(\id\ot c)(c\ot \id)=(\id\ot c)(c\ot \id)(\id\ot c)\qquad \text{ in }\End(V\ot V\ot V).
\end{equation}
In such a case, the pair $(V,c)$ is called a braided vector space.

Given a braided vector space $(V,c)$, Faddeev, Reshetikhin and Takhtajan \cite{FRT} 
introduced a method, \textit{the FRT-construction} 
for short, to
construct a coquasitriangular bialgebra $A(c)$ such that $V$ is an $A(c)$-comodule and $c$ is a morphism of 
$A(c)$-comodules.
By the very definition, this bialgebra is universal with such properties. 
Besides, it turns out that the category ${}^{A(c)}\M$ of left  $A(c)$-comodules
is braided monoidal.
Notice that if $V\neq 0$ then
$A(c)$ is never a Hopf algebra: 
suppose the contrary and write $\SS$ for the antipode. Since
$A(c)=\oplus_{n\in\N_0}A(c)^n$ is graded in non-negative degrees and
generated by comatrix elements $t_i^j$,
\[
1=\epsilon(t_i^i)=\sum_kt_i^k\SS(t_k^i)
\in\bigoplus_{n>0}A(c)^n,\]
one gets a contradiction.
In the trivial 
example given by $\tau$= the flip map in $\Bbbk^n$, the FRT-construction
yields the coordinate affine ring $\O(\Mq_n)$ on $n\times n$ matrices over $\Bbbk$, 
and one needs to localize the commutative algebra on the determinant in order to 
obtain the Hopf algebra
$\O(\GL_n)$. In general, the {\em abelianization} of the FRT-construction
gives the bialgebra $\O(\End(c))$, which is the ring of coordinate functions on the endomorphisms 
$f$ of $V$ such that $f\ot f$ commutes with $c$. If one could localize the FRT-construction
and get a Hopf algebra $H(c)$, then one would have a surjective map
$H(c)\twoheadrightarrow \O(\Aut(c))$, that is, a quantum group
 much larger than
the "classical" automorphism group of the braiding.
In general, by 
\cite[Lemma 3.2.9]{Sch} (see \cite{T1}
for a review) in case the braiding $c$ is rigid there exists a coquasitriangular Hopf algebra 
$\mathcal{H}(c)$ associated with $(V,c)$ satisfying a universal property: 
$V \in {}^{\mathcal{H}(c)}\M$ with a certain comodule structure map $\lambda$ and  
if $B$ is a coquasitriangular bialgebra such that $V \in {}^{B}\M$ with comodule structure 
$\lambda_{B}:V \to B\ot V$, then there exists a coquasitriangular bialgebra map 
$f:\mathcal{H}(c) \to B$ such that $\lambda_{B}= (f\ot \id)\lambda$.
Furthermore, 
by \cite[Lemma 3.2.11]{Sch}, the 
Hopf algebra $\mathcal{H}(c)$ is generated as algebra by elements 
$\{\frak{t}_{i}^{j},\ \frak{u}_{i}^{j}\}_{1\leq i,j \leq n}$ satisfying 
\begin{equation}\label{eq:relationsPeter}
\sum_{k,\ell}c_{ij}^{k\ell}\frak{t}_k^r \frak{t}_\ell^s=
\sum_{k,\ell}\frak{t}_i^k \frak{t}_ j^\ell c_{k\ell}^{rs},\qquad \text{ and }\qquad 
\sum_{k=1}^{n} \frak{u}_{i}^{k} \frak{t}_{k}^{j} = \delta_{i}^{j}=  \sum_{k=1}^{n} \frak{t}_{i}^{k} \frak{u}_{k}^{j}. 
\end{equation} 
The coalgebra structure is given by $\com(\frak{t}_{i}^{j}) = \sum_{k=1}^{n} \frak{t}_{i}^{k} \ot \frak{t}_{k}^{j}$,
$\eps(\frak{t}_{i}^{j}) = \delta_{i}^{j}$ and $\com(\frak{u}_{i}^{j})= \sum_{k=1}^{n} \frak{u}_{k}^{j} \ot \frak{u}_{i}^{k}$, 
$\eps(\frak{u}_{i}^{j})=\delta_{i}^{j}$. Moreover, 
one has that $\SS_{\mathcal{H}(c)}(\frak{t}_{i}^{j}) = \frak{u}_{i}^{j}$ for all $1\leq i,j\leq n$.
Note that, since $\mathcal{H}(c)$ is coquasitriangular, the square of the antipode is an inner automorphism,
and as a consequence, the antipode and all its powers are defined on the generators $\frak{t}_{i}^{j}$, $\frak{u}_{i}^{j}$. 
The  comodule category ${}^{\mathcal{H}(c)}\M$ is the one generated by $V$ and $V^*$, and in general
the map $A(c)\to \mathcal{H}(c)$ needs not to be injective (see example in Subsection
\ref{EjemploPeter}). 

In this paper we consider  the following 3-step problem: given a finite-dimensional 
rigid braided vector space $(V,c)$,
\begin{enumerate}
 \item[$(a)$] find a ''quantum determinant" for the FRT-construction $A(c)$,
 \item[$(b)$] prove that
the localization $H(c)  = A(c)[D^{-1}]$ of $A(c)$ at the quantum determinant is a Hopf algebra,
\item[$(c)$] prove that $H(c)  \simeq \mathcal{H}(c)$.
\end{enumerate}
In Subsection \ref{subsec:quantumdeterminant}, we introduce a method for finding a 
quantum determinant 
associated with a rigid solution of the braid equation.
Two of our main results, Theorem \ref{teomain} and Theorem \ref{teomain-new},
give sufficient conditions to ensure the existence, and a concrete way to compute it,
of a group-like element $D\in A=A(c)$, such that $D$ is normal
 in $A$ and, under certain conditions, the localization on $D$ is a Hopf algebra $H(c)$.
 Moreover, our proof yields 
an explicit formula for the antipode.
 Finally, we show in Corollary \ref{cor:H=HPeter} that 
 $H(c)$ is isomorphic to the universal coquasitriangular Hopf algebra 
$\mathcal{H}(c)$ associated with $(V,c)$. 
In this way, we obtain a realization of $\mathcal{H}(c)$ as a localization of $A(c)$.

As a classical motivation of this problem one can mention the famous work of Y. Manin \cite{M}, see also \cite{M2},
where the author introduces two operations 
$\bullet$ and $\circ$ on quadratic algebras, interpreted as  internal tensor products,
and proves that the internal end$(A)=A^!\bullet A$ of a quadratic algebra $A$
is always a bialgebra, recovering some remarkable examples such as the quantum function algebra $\Mq_q(2)$.
The problem of finding quantum determinant is present in this work, introducing 
what Manin calls a \textit{quantum grasmannian algebra} (qga)  in \cite{M}, or a \textit{Frobenius quantum space} (Fqs) 
in \cite{M2}, where a "volume form" 
 plays a crucial role.
The definition of a qga, or a Fqs, assures the existence of a group-like element that is the natural
candidate for a quantum determinant, but the problem of
finding the antipode (or even to prove its existence) remains open.

In \cite{H}, Hayashi constructed quantum determinants for multiparametric 
quantum deformations of $\O(\SL_{n})$, $\O(\GL_{n})$, $\O(\SO_{n})$, $\O(\Or_{n})$ and $\O(\Sp_{2n})$, inverting
{\em all} group-like elements in a given quasitriangular bialgebra, and showing that the
ending result is a Hopf algebra. To define the quantum determinants, qga's are considered
for the deformations of the
classical examples.
The idea of considering 
\textit{quantum exterior algebras (qea)} is also present in the work 
of Fiore \cite{F}, where
 the author defines quantum determinants for the quantum function algebras 
$\SO_{q}(N)$, $\Or_{q}(N)$, and $\Sp_{q}(N)$,
 which are defined through (a quotient of) the FRT-construction, by means of the
 coaction of these on a volume element. This is where the quantum determinant 
comes into (co)action.
 More generally, qea's and quantum determinants appear in the work of Etingof, Schedler 
and Soloviev \cite{ESS} 
 as universal objects associated with the exterior algebra when considering set-theoretical
 (involutive) solutions to QYBE's. All 
 quantum determinants appearing in this way should be central.
 Nevertheless, we found an example 
 that this might not be the case, see Subsection \ref{ex:involutive-non-central}.
 
Motivated by the results in \cite{ESS}, the definition of the qga and the 
quantum exterior algebras, and properties of  
the {\em Nichols algebra} associated with a rigid solution of the braid equation,
in these notes we introduce certain class of graded connected algebras extending Manin's definition of Fqs, see 
Definition \ref{def:finite-Nichols-type}, 
that enable us not only to consider volume elements and 
prove the existence of a quantum
 determinant, but also to find an explicit formula for the natural candidate 
of the
{\em antipode} in the FRT-construction, localized at the quantum determinant.

These qga's defined 
by Hayashi and the quantum exterior algebras considered by Fiore are all quadratic. In general,
for a given braiding,  there is no quadratic qga, but still there might be a 
finite-dimensional Nichols algebra associated with it. 
As a consequence, our method still apply in this case, see example in
Subsection \ref{ejchino}.

Quantum determinants are intensively studied in the literature as
 the classical problem of
  defining the determinant of a matrix with non-commutative entries, and because
 they also give a way to construct 
 new examples of quantum groups, 
 see for example \cite{M2}, \cite{KL}, \cite{PW}, \cite{ER}, \cite{CWW}, \cite{JZ}, \cite{JoZ}, \cite{KKZ} and references therein.
 It is worth to mention that in the work on quantum determinants by Etingof and
Retakh \cite{ER}, the existence of formulas with
"quantum minors" is considered. In our approach, the 
existence and concrete formulas for these "minors" emerge clearly.

Another features of the procedure to find quantum determinants are the following: 
given a finite-dimensional Nichols algebra $\B$, the method provides a Hopf algebra  
$H$ such that $\B$ is a braided Hopf algebra in the braided category of left $H$-comodules. 
It also gives families of
Hopf algebras generated  by simple subcoalgebras. Finite-dimensional quotients of these kind of Hopf algebras
are of interest in the classification program of finite-dimensional complex Hopf algebras by means of the 
generalized Lifting Method, see for example \cite{AC}, \cite{GJG}.

The paper is organized as follows. 
In Section \ref{sec:ABC} we recall the FRT-construction and 
the definition of the Nichols algebra associated with a
braided vector space. In Section \ref{sec:quantumdeterminant}, we introduce the method for
 finding quantum determinants and "quantum cofactor formulas", proving
 our main results Theorems \ref{teomain}, \ref{teomain-new} and Corollary \ref{cor:H=HPeter}. Finally, we 
illustrate our contribution with several examples, including cases where
the determinant is not central, and 
 quantum function algebras from Fomin-Kirillov algebras
 associated with the symmetric groups on three letters.

\subsection*{Acknowledgments}
We want to thank Peter Schauenburg for answering 
all our questions, 
together with the clearest example in each case, 
and Chelsea Walton for many hints and suggestions.
We also thank the referee for the careful reading of our manuscript,
and for the comments and suggestions that helped us to improve the presentation. In particular,
the argument used in the introduction to show that $A(c)$
is never a Hopf algebra is due to her/him.

\section{Preliminaries: the $A$ $\B$ $c$}\label{sec:ABC}
In this section we give the definitions and basic properties of the FRT-construction and Nichols 
algebras, and recall known results that are needed for our construction.

Throughout the notes,  $\Bbbk$ denotes an arbitrary field. We use 
the standard conventions for Hopf algebras and write $\Delta$, $\varepsilon$
and $\SS$ for the coproduct, counit and antipode, respectively. We also use Sweedler's notation $\Delta(h)=h_{(1)}\ot
h_{(2)}$ for the comultiplication. 
Given a bialgebra $A$, the category of finite-dimensional left $A$-comodules is
 denoted by $\ ^{A}\M$.
The readers are referred to \cite{Ra} for further details on
the basic definitions of Hopf algebras.

\subsection{The FRT-construction: $A(c)$ \label{sectionFRT}}
In this subsection, we follow \cite{LR}.
Let $(V,c)$ be a finite-dimensional braided vector space and fix
$\{x_i\}_{i=1}^n$ a basis of $V$. 
Write $\{x_{i}^{*}\}_{i=1}^{n}$ for the basis of $V^{*}$ dual to $\{x_{i}\}_{i=1}^{n}$. 
Recall that a solution of the
braid equation 
$c$ is \textit{rigid},
if the map $c^{\flat} : V^{*} \ot V \to V \ot V^{*}$ given by  $c^{\flat}(f\ot x) = \sum_{i=1}^{n} (\ev\ot\id\ot\id)(f \ot c(x\ot x_{i})\ot x_{i}^{*})$ is invertible.

Let $C=\End(V)^*$ be the coalgebra linearly spanned 
by the matrix coefficients $\{t_i^j\}_{1\leq i,j\leq n}$. Then,
$V$ has a natural left $C$-comodule structure. Note that, as $C\cong M_n(\Bbbk)^*\cong 
V\ot V^*$ these generators are induced by the basis $\{x_i\}_{i=1}^{n}$, 
via the correspondence $t_i^j\leftrightarrow x_i\ot x^*_j$.
The coalgebra structure is given by
\begin{equation}\label{eq:coalgstrucC}
\Delta(t_i^j)=\sum_{k=1}^nt_i^k\ot t_k^j,\qquad \eps(t_{i}^{j})=\delta_{ij}\qquad\text{ for all }1\leq i,j\leq n,
 \end{equation}
and $V$ is a (left) $C$-comodule 
by setting
\[
\lambda (x_i)=\sum_{j=1}^{n}t_i^j\ot x_j\qquad \text{ for all }1\leq i \leq n.
\]
Write $TC$ for the tensor algebra of $C$. Extending as algebra maps the comultiplication and the counit of $C$ to $TC$, the latter
becomes a bialgebra and 
$V\ot V$ is a (left) $TC$-comodule. In general, 
a linear map $c:V\ot V\to V\ot V $ is not necessarily $TC$-colinear. Actually,
if one consider the difference of the two possible compositions in the following diagram,
computed in the basis $\{x_i\ot x_j\}_{i,j}$, one gets
$$
\xymatrix{
x_i\ot x_j\ar@{|->}[d]\ar@/^3ex/@{|->}[rrr]&V\ot V\ar[r]^c\ar[d]^{\lambda}& V\ot V\ar[d]^{\lambda}&\sum_{k,\ell}c_{ij}^{k\ell}x_k\ot x_\ell\ar@{|->}[d]\\
\sum_{k,\ell}t_i^kt_j^\ell\ot x_k\ot x_\ell
\ar@/_3ex/@{|->}[drrr]
&TC\ot (V\ot V)\ar[r]^{\id\ot c}& TC\ot (V\ot V)&
\sum_{r,s,k,\ell}c_{ij}^{k\ell}\ t_k^rt_\ell^s\ot x_r\ot x_s \\
&&&\sum_{k,\ell,r,s}t_i^kt_ j^\ell c_{k\ell}^{rs}\ot x_r\ot x_s
} $$
where the coefficients $c_{ij}^{k\ell}$ are defined by the 
equality $c(x_i\ot x_j)=\sum_{kl}c_{ij}^{k\ell}x_k\ot x_\ell$.
Hence, one arrives naturally at the following definition:

\begin{definition}\cite{FRT}
The FRT-construction (or universal quantum semigroup) 
for $(V,c)$ is the $\Bbbk$-algebra $A=A(c)$ generated by the elements $\{t_i^j\}_{1\leq i,j\leq n}$,
satisfying the following relations:
\begin{equation}\label{eq:FRT}
\fbox{\fbox{$
\sum_{k,\ell}c_{ij}^{k\ell}t_k^rt_ \ell^s=
\sum_{k,\ell}t_i^kt_ j^\ell c_{k\ell}^{rs}
$}}\hskip 1cm \forall\ 1\leq i,j,r,s\leq n.
 \end{equation}
 \end{definition}
It is well-known that $A(c)$ is a bialgebra with comultiplication and counit determined by \eqref{eq:coalgstrucC}, which 
satisfies
a universal property:
the map $\lambda: V\to A(c) \ot V$ equips $V$ with the structure of a left 
comodule over $A(c)$
such that 
the map $c$ becomes a comodule map.
If $\mathcal{A}$ is another bialgebra coacting on $V$ via a linear map $\lambda'$ such
that $c$ is $\mathcal{A}$-colinear, then there exists a unique bialgebra morphism 
$f : A(c) \to \mathcal{A}$ such that
$\lambda' = (f\ot \id_{V})\lambda$.

\begin{rem}\label{lem:lambda-colineal} 
Let $V$ be a finite-dimensional $\Bbbk$-vector space, 
$c \in \End(V\ot V)$ and  $A=A(c)$.
For $n\geq 2$, the linear map given by
$
c_k:=\id_{V^{\ot k-1}}\ot c\ot\id_{V^{n-k-1}}:V^{\ot n}\to V^{\ot n}
$
is $A$-colinear. That is, the comodule map
$\lambda:V^{\ot n}\to A\ot V^{\ot n}$ satisfies that
\[
\lambda c_k=(\id_A\ot c_k)\lambda\qquad \text{ for all }1\leq k\leq n-1.\]
\end{rem}
\begin{proof}
This follows from the fact that $c$ is $A$-colinear and the category of 
$A$-como\-dules is tensorial.
\end{proof}

It is well-known that if $c$ satisfies the braid equation, then $A=A(c)$ is a
coquasitriangular bialgebra, that is, there exists a convolution-invertible bilinear map
$r:A\times A\to  \Bbbk$ satisfying

\[
\begin{array}{crcl}
(CQT1)\qquad&r(ab,c)&=&r(a,c_{(1)})r(b,c_{(2)})
\\
(CQT2)\qquad&r(a,bc)&=&r(a_{(2)},b)r(a_{(1)},c)
\\
(CQT3)\qquad&r(a_{(1)},b_{(1)})a_{(2)}b_{(2)}&=&b_{(1)}a_{(1)}r(a_{(2)},b_{(2)})
\end{array}
\]
This map is uniquely determined by
\[
r( t_i^k,t_j^\ell)=
c_{ji}^{k\ell}\qquad \text{ for all }1\leq i,j,k,\ell\leq n.
\]

\begin{rem}
\label{rmk:D-coquasi}
\begin{enumerate}
\item[$(a)$] The first two conditions
say that for any group-like element $D$,
the maps $r(D,-)$, $r(-,D): A\to \Bbbk$ are algebra maps. 
\item[$(b)$] The
last condition can be express by the equality $r*m=m^{op}*r$. Moreover, on   
$a=t_j^r$ and $b=t_ i^s$, it reads
\[
\sum_{k,\ell}r(t_j^k, t_i^\ell) t_k^rt_\ell^s
=
\sum_{k,\ell} t_i^\ell t_j^kr( t_k^r,t_\ell^s),
\]
that is,
\[
\sum_{k,\ell}c_{ji}^{k\ell} t_k^rt_\ell^s
=
\sum_{k,\ell} t_i^\ell t_j^kc_{lk}^{rs}.
\]
Conditions $(CQT1)$ and $(CQT2)$ say that $r$ is determined by the values of 
$r$ on generators, so it extends to the tensor algebra; $(CQT3)$ says that $r$ descends to $A$.

\item[$(c)$]
For a group-like element $D$, 
 $(CQT3)$ gives a commutation rule:
\[
r(D,b_{(1)})Db_{(2)}=b_{(1)}D\ r(D,b_{(2)})\qquad \text{ for all }b\in A\]

\item[$(d)$] The category $^{A}\M$ is braided with $c_{M,N}:M\ot N \to N \ot M$ given by 
$$c(m\ot n) = r(m_{(-1)}, n_{(-1)})\ n_{(0)} \ot m_{(0)}\qquad \text{ for all }M, N \in \ ^{A}\M.$$
\end{enumerate}
\end{rem}
A stronger result than the commutation rule above is due to Hayashi 
and holds for any coquasitriangular bialgebra.

\begin{lem}\label{lemaHayashi}
\cite[Theorem 2.2]{H} Let $A$ be a coquasitriangular bialgebra.
For any group-like element $g\in A$, there is
a bialgebra automorphism $\mathfrak J_{g}:A\to A$ given by 
$
\J_{g}(a)=r(a_{(1)},g)a_{(2)}r^{-1}(a_{(3)},g)$
such that
\[ga=\J_{g}(a)g\qquad \text{ for all } a\in A.\]
\end{lem}

\begin{ex}
Let  $X$ be a set and $s:X\times X \to X\times X$ a set-theoretical solution
of the braid equation, that is $s$ satisfies
\[
(s\times \id_X)(\id_X\times s)
(s\times \id_X)=(\id_X\times s)
(s\times \id_X)(\id_X\times s).
\]
For $x,y, a,b \in X$, let $z,t,u,v \in X$ be
such that 
$(z,t)=s(x,y)$, and $s(u,v)=(a,b)$. 
Let $V=\Bbbk X$ be the $\Bbbk$-vector space linearly spanned by the elements of $X$ and let $c$ be the 
linearization of $s$.
Then, the set of equations for the corresponding 
FRT-construction on $(V,c)$ is
\[
\fbox{$t_x^ut_y^v=t_z^at_t^b$}
\]
In particular, for the flip solution $\tau(x,y)=(y,x)$ on a finite set $X=\{x_{1},\ldots, x_{n}\}$,
we have that $t_x^bt_y^a=t_y^at_x^b$; in other words,
$
A(\tau)=\O (\Mq_n)$.
This is  not a Hopf algebra, but if one consider the element in $A$ given by the 
usual determinant
\[
D:={\det}_n
=\sum_{\sigma\in \s_n}(-1)^{\ell(\sigma)} t_{\sigma (1)}^1\cdots t_{\sigma (n)}^n,
\] 
then the localization on $D$ is the Hopf algebra $A(\tau)[D^{-1}]=\O(\GL_n)$. 
We will generalize this construction for nontrivial
examples.
\end{ex}

\begin{rem} Note that, since \eqref{eq:FRT} is homogeneous, 
 $A(c)=A(qc)$ for all $0\neq q\in \Bbbk$. Also, if $c$ is invertible, then $A(c)=A(c^{-1})$.
\end{rem}

\subsection{Nichols algebras: $\B$}
Let $(V,c)$ be
a braided vector space.
The braid group  
\[
\mathbb{B}_n=\langle \tau_{1},\ldots, \tau_{n-1}|\ \tau_{i}\tau_{j}=\tau_{j}\tau_{i}, \
\tau_{i+1}\tau_{i}\tau_{i+1}=\tau_{i}\tau_{i+1}\tau_{i}, 
\text{ for }
1\leq i\leq n-2\text{ and }j\neq i\pm1\rangle
\]
acts
on $V^{\ot n}$ via $\rho_{n}:\mathbb{B}_{n} \to \GL(V^{\ot n})$ with $\rho_{n}(\tau_{i})=c_{i} 
=\id_{V^{\ot i-1}}\ot c\ot\id_{V^{n-i-1}}:V^{\ot n}\to V^{\ot n} $. Using the
Matsumoto
(set-theoretical) section from the symmetric group $\mathbb{S}_{n}$ to $\mathbb{B}_n$:
\[
M:\mathbb{S}_n\to \mathbb{B}_n,\qquad
(i,i+1)\mapsto \tau_i,\qquad \text{for all }1\leq i\leq n-1,\]
one can define the quantum symmetrizer
$QS_n:V^{\ot n}\to V^{\ot n}
$ by
\[
QS_{n} = \sum_{\sigma \in \s_{n}}\rho_{n}(M (\sigma)) \in \End(V^{\ot n} ).
\]
For example $QS_{2} = \id + c$, and 
\[
QS_{3}=\id+c\ot \id+\id\ot c+(\id\ot c)(c\ot \id)+(c\ot \id)(\id\ot c)+(c\ot \id)(\id\ot c)(c\ot \id).
\]

The Nichols algebra associated with $(V,c)$
is the quotient of the tensor algebra $TV$  by the homogeneous ideal
\[
\mathcal{J}=\bigoplus_{n\geq 2 } \Ker QS_{n},
\]
 or equivalently,
$\B(V,c):=\oplus_n\Im(Q{S}_n)$. In particular,
$\B(V,c)$ is a graded algebra.
Note that $\B^0(V,c)=\Bbbk$,
$\B^ 1(V,c)=V$ and $\B^2(V,c)=(V\ot V)/(\Ker (\id + c))$.

There are several equivalent definitions of the Nichols algebra associated with 
$(V,c)$, 
each of them 
particularly useful for different purposes.
For more details, see \cite{A}.

\begin{prop}\label{prop:B-A-comodulo}
The Nichols algebra  $\B(V,c)$ is an $A(c)$-comodule algebra.
\end{prop}

\begin{proof}
By Remark
\ref{lem:lambda-colineal}, we have that $c_{k}$ is $A(c)$-colinear, which implies that 
$QS_{n}$ is an $A(c)$-comodule map. Thus, $\Ker QS_{n}$ is an
 $A(c)$-subcomodule of $V^{\ot n}$ for all $n\geq 2$. Hence,
taking the quotient module $\mathcal{J}$ defines an $A(c)$-comodule structure on 
$\B(V,c) = TV/\mathcal{J}$.
\end{proof}

Nichols algebras are a key ingredient in the classification of finite-dimensional pointed 
Hopf algebras and there is extensive literature covering the problem of finding
finite-dimen\-sional Nichols algebras. 
If the Nichols algebra is finite-dimensional and the braiding is rigid,
then special features arise. These properties
guide us to make a general construction that motivates
the definition of "weakly graded-Frobenius algebra"
that is the core of next section.

\section{The quantum determinant and the antipode formula}\label{sec:quantumdeterminant}
In this section we introduce a method for finding a quantum determinant 
associated with a rigid solution of the braid equation (and additional assumptions), and
prove our main results in Theorems \ref{teomain}, \ref{teomain-new} and Corollary \ref{cor:H=HPeter}.

\subsection{The quantum determinant} \label{subsec:quantumdeterminant}
The following definition extends the notion of Frobenius quantum space 
introduced by Manin in \cite[\S 8.1]{M2}. As in
\textit{loc. cit.}, we use it to define quantum determinants, to establish quantum Cramer 
and Lagrange identities, and to produce categorical dual objects.
\begin{defi}\label{def:finite-Nichols-type} 
Let $\mathcal{A}$ be a bialgebra and $V \in\ ^{\mathcal{A}}\M$.
An $\mathcal{A}$-comodule algebra $\B$ is called a
{\bf weakly graded-Frobenius} (WGF) algebra for $\mathcal{A}$ and $V$ if the following conditions are satisfied:
\begin{itemize}
\item[(WGF1)] $\B$ is an $\N$-graded $\mathcal{A}$-comodule algebra, that is
 $\B=\underset{n\geq 0}\bigoplus\B^n$, 
$\lambda(\B^n)\subseteq \mathcal{A}\ot \B^n$, where
$\lambda:\B\to  \mathcal{A}\ot \B$ is the structure map,
and $\B^n\cdot \B^m\subseteq \B^{n+m}$ for all $n,m\geq 0$;
\item[(WGF2)]  $\B$ is connected (i.e. $\B^0=\Bbbk$) and $\B^1=V$ as $\mathcal{A}$-comodules;
\item[(WGF3)] $\dim_\Bbbk\B<\infty$ and $\dim_\Bbbk\B^{top}=1$, where
$top=\max\{n\in\N : \B^n\neq 0\}$;
\item[(WGF4)] the multiplication induces
non-degenerate bilinear maps
\[
\B^1\times\B^{top-1}\to \B^{top},\qquad
\B^{top-1}\times\B^1\to \B^{top}.
\]
\end{itemize}
\end{defi}

Some remarks are in order:

$(i)$ Let $\mathcal{A}, \mathcal{A}'$ be bialgebras and let $\B$ a
WGF-algebra for $\mathcal{A}$ and $V$.
If $f:\mathcal{A}'\to \mathcal{A}$ is a bialgebra map, then $\B$ is also a
WGF-algebra for $\mathcal{A}'$. 

$(ii)$
 Let $\mathcal{A}$ be a bialgebra and $(V,c)$ a braided vector space. If 
$V \in\ ^{\mathcal{A}}\M$ is
 such that $c$ is $\mathcal{A}$-colinear
 and $\B$ is a
WGF-algebra for $\mathcal{A}$ and $V$, then
 the universal property of $A(c)$ determines a unique bialgebra map
$f:A(c) \to  \mathcal{A}$. Consequently 
$\B$ is a
WGF-algebra for $A(c)$ and $V$. In this case, 
$\B$ is directly associated with the braided vector space $(V,c)$. For short, 
we say that $\B$ is a
WGF-algebra for $A(c)$.

$(iii)$
A finite-dimensional graded algebra $\B=\underset{n\geq 0}
\bigoplus\B^n$ with
$\B^{0}=\Bbbk$ is called {\em graded-Frobenius} (GF) if there exists $p \in \N$ such that
$\dim \B^{p} = 1$, $\B^{p+j} = 0$ for $j > 0$ and the multiplication 
$\B^{ j} \times \B^{p-j} \to \B^{p}$ is
non-degenerate {\em for  all} $j$ with $0 \leq j \leq p$.
For instance any finite-dimensional graded connected
Hopf algebra in the category of Yetter-Drinfeld modules over a Hopf algebra
$H$ is GF, see \cite{N} and \cite[\S 3.2]{AG1} for more details.

$(iv)$ Let $(V,c)$ be a finite-dimensional rigid braided vector space and let $\B=\B(V,c)$
 be the Nichols algebra associated with it. 
If $\dim_k\B<\infty$, then by Proposition \ref{prop:B-A-comodulo}, the very definition of
 Nichols algebra and $(iii)$ above, 
it follows that $\B$ is a GF-algebra and hence 
a WGF-algebra for $A(c)$.  In this way, the theory of Nichols algebras
 provides plenty of  examples that are not necesarily quadratic,
 nor $N$-homogeneous.

$(v)$
One can easily give examples of WGF-algebras that are not 
GF  by adding to a GF-algebra
some elements in intermediate degrees with zero products, but these examples
are artificial in the sense that they do not occur naturally from the data $(V,c)$.
 However, given  an   algebra $\B$, it is in general
a difficult task to check whether or not it is a Nichols algebra: one should also  care
about   the coalgebra structure, verify that it is generated in degree one and there are no primitive
 elements of degree bigger than one. But for our pourposes, the only property that we need from the
algebra $\B$ is just part of the definition of graded Frobenius, and this is easy to 
check in examples. For this reason, we decide to extend the notion 
from GF to WGF, even though the only (no artificial) examples that we have 
are already GF. As a matter of example, concerning the Foming-Kirillov 
algebras, for the (known to be) finite dimensional ones, their Hilbert polynomials
were known much before we knew they were Nichols algebras, see for example \cite{FK}, \cite{MS}, \cite{AG2}, \cite{Gr} and \cite{GGI}.

$(vi)$ It is known that there are plenty of 
examples of braided Hopf algebras
in positive characteristic that are not Nichols algebras (e.g. they failed to have all 
primitive  elements in degree one). To the best of our knowledge, 
there are no examples in characteristic zero of graded connected finite-dimensional 
braided Hopf algebras that are not Nichols algebras.

$(vii)$ The conditions in (WGF4) appeared in \cite{M2}, related to involutive solutions of the 
QYBE (thus the corresponding $c$ is a symmetry) and in \cite{Gu},
related to Hecke-type solutions. It is known that in both cases the quantum
exterior algebras are Nichols algebras, thus this Definition generalizes \cite{M2,Gu}.

Fix a braided vector space $(V,c)$ and let $A=A(c)$ be the bialgebra given by the FRT-construction associated with $(V,c)$. 
The existence of a weakly graded-Frobenius algebra $\B$ for $A$ allows not only to
define a \textit{quantum determinant} for $A$, but also to give an explicit formula
for the antipode. We begin with the definition of a quantum determinant associated with $\B$.
Note that our definition is consistent with quantum (homological) determinants defined previously by other authors, see
for example \cite{M2}, \cite{JoZ}, \cite{KKZ}, \cite{CWW}.

\begin{defi} Let $\B$ be a weakly graded-Frobenius algebra for $A$ and write 
$\B^{top}=\Bbbk \b$ for some $0\neq \b\in \B$.  
We call such an element a \textit{volume element} for $\B$.
Since $\B^{top}$  is an $A$-subcomodule, we have that the coaction on $\b$ equals 
$\lambda(\b)= D\ot \b$ for some group-like element $D\in A$. 
We call this element $D$ {\em the quantum determinant in $A$ associated with $\B$}.
\end{defi}

\noindent Note that $D\in G(A)$ is independent of the scalar multiple of $\b$.

\begin{ex}\label{ex:exterior} Consider the braiding $c=-\tau$ 
on an $n$-dimensional space $V$. Then $A(-\tau)=\O(\Mq_{n})$ and  $\B(V,c) = \Lambda V$ is a left
$\O(\Mq_{n})$-comodule algebra. If
$\{x_1\cdots,x_n\}$ is a basis of $V$, then one may take
$\b=x_1\wedge\cdots \wedge x_n \in \Lambda^{n} V$. 
In this case,
$$D=\sum_{\sigma \in \mathbb{S}_n}(-1)^{\ell(\sigma)}
t_{\sigma(1)}^1\cdots t_{\sigma(n)}^n$$ is 
given by the usual determinant.
\end{ex}

\begin{notation}\label{not:Tij}
Let $\{x_1,\dots,x_n\}$ be a basis of $V$. 
Since by assumption  the multiplication $\B^1\times\B^{top-1}\to \B^{top}=\Bbbk \b$ is non-degenerate,
there exists a basis of $\B^ {top-1}$, say
$\{\om^1,\dots,\om^n\}\in\B^ {top-1}$, such that
\[
x_i\om^j=\delta_i ^j \b\quad \in \B^{top}.
\]
For $1\leq i,j\leq n$, we define the elements $T_i^j\in A$ by the equality
\[
\lambda(\om^i)=\sum_j T^i_j\ot \om^j \qquad \text{ for all }1\leq i \leq n.
\]
It is easy to check that  
$\Delta (T_j^i)=\sum_{k=1}^{n}T_k^i\ot T_j^k$ 
and $\eps(T_{i}^{j}) = \delta_{i}^{j}$ for all $1\leq i,j\leq n$.
\end{notation}

\begin{ex} 
Consider the braiding $c=-\tau$ on $V\ot V$
as in Example \ref{ex:exterior} above. Then, the elements  $w^j=(-1)^{i+1}x_1\wedge\cdots \wh {x_j}\cdots \wedge x_n$ give a "dual basis"
with respect to $\{x_1,\dots,x_n\}$
and the volume form
$\mathfrak b=
x_1\wedge\cdots  \wedge x_n$.
\end{ex}

Next we generalize the formula when expanding a determinant by a row using minors:

\begin{prop}\label{propFila}
The following formula holds in $A(c)$:
\begin{equation}\label{eq:A(c)-D}
\sum\limits_{k=1}^n t_i^kT^j_ k = \delta_ i^j D \qquad \text{ for all }1\leq i,j\leq n.
 \end{equation}
\end{prop}

\begin{proof}
Using the fact that $\{x_i\}_{1\leq i\leq n}$ and $\{w^j\}_{1\leq j\leq n}$ are dual bases
with respect to the multiplication, that is
$x_i\om^j=\delta_i^j\mathfrak b$ for all $1\leq i,j\leq n$, by
the comodule structure on $\B$ we get that 
\begin{align*}
\lambda(\delta_i^j\b) &=
\delta_i^jD\ot \b=\lambda(x_i\om^j)
=\lambda(x_i)\lambda(\om^j)\\
& =\sum_{k,\ell}t_i^kT^j_\ell\ot x_k\om^\ell
=\sum_{k,\ell}t_i^kT^j_\ell\ot \delta_k^\ell\b
=\sum_{k}t_i^kT^j_k\ot \b.
 \end{align*}
\end{proof}

\begin{ex} For $M\in \Mq_{n}(\Bbbk)$, let $\Cof(M)$ be the $(n\times n)$-matrix whose $( i,j)$-entry is the $ij$-minor.   
For $c=-\tau$, Proposition \ref{propFila} is nothing else than the well-known fact 
\[
M\cdot \Cof(M)^t=\det(M)I \hskip 1cm \forall\ M\in \Mq_n(\Bbbk).
\]
\end{ex}

\subsection{Main results}
In this subsection we prove our main theorem. We begin by introducing a Hopf algebra 
associated with the quantum determinant and give some properties of 
its category of finite-dimensional left comodules.
For the rest of this subsection, we fix 
a finite-dimensional rigid braided vector space  $(V,c)$
and assume there exists a weakly graded-Frobenius
algebra $\B$ for $A$. We write $D$ for the quantum determinant and $\b$ for the volume element. 
By Lemma \ref{lemaHayashi}, we know that there exists 
an automorphism $\J:=\J_{D} \in \Aut(A)$ associated with 
the quantum determinant $D$ such that $Da = \J(a)D$ for all $a\in A$.

\begin{defi}\label{def:localization}
Let $A$ be a $\Bbbk$-algebra and $D$ a non-zero element in $A$. We define the \textit{localization} of 
$A$ in $D$ as a pair $(H,\iota)$, where $H$ is a $\Bbbk$-algebra and $\iota: A \to H$ is an algebra map 
that satisfies the following universal property: for any algebra map $f:A\to B$ such that $f(D)$ is invertible 
in $B$, there exists a unique algebra map $\bar{f}: H \to B$ such that $\bar{f}\circ \iota = f$; i.e. the following
diagram commutes
$$ \xymatrix{A \ar[rr]^{\iota} \ar[rd]_{\forall\ f \text{ s.t.} \atop f(D) \text{ invertible}}& & H \ar@{-->}[dl]^{\exists\ !\ \bar{f}} \\
& B &}
$$
We call $\iota: A \to H$ the \textit{canonical map}. By the universal property above, 
it follows that the localization, if it exists, is unique up to isomorphism. We denote it 
as $H=A[D^{-1}]$, if no confusion arises. 
\end{defi}

\begin{remark}
The 
 localization of a bialgebra with respect to a group-like element $D$ always exists in 
the following sense:
If $A$ is a $\Bbbk$-algebra one can always consider the polynomial algebra in one
 indeterminate $\Bbbk[x_0]$ and the free product
$A*\Bbbk[x_0]$. It has the universal property that 
 given a $\Bbbk$-algebra map $f:A\to B$
and $b_0\in B$, then there exists a unique $\Bbbk$-algebra 
map $\Phi:A*\Bbbk[x_0]$ such that $\Phi|_A=f$ and
 $\Phi(x_0)=b_0$.

Now, given an element $D\in A$, one can consider the two-sided ideal $J$ generated by 
$x_0D-1$ and $Dx_0-1$, and define $A[D^{-1}]:=(A*\Bbbk[x_0])/J$. If $f:A\to B$ is an algebra map such that 
$f(D)=s_0$ is invertible, then one can consider $s_0^{-1}\in B$ and define $\Phi:A*k[x_0]\to B$
by $\Phi|_A=f$ and $\Phi(x_0)=s_0^{-1}$. This map satisfy $\Phi(x_0D-1)=0=\Phi(Dx_0-1)$, so,
it induces an algebra map on the quotient. In other words, $H:=A*\Bbbk[x_0]/J$ satisfies the
 universal property.

The only problem that one can face is that maybe $J$ is not a proper ideal. If $J=A*\Bbbk[x_0]$
(e.g. if $D=0$ then $J=\langle 1\rangle$)  then $H$ is the zero algebra, 
and one has
$1=0$ in $H$.

For a counitary bialgebra $A$ and a nonzero group-like element $D\in A$
 one has the advantage that $\varepsilon(D)=1$ (by counitarity) and 
 $A*\Bbbk[x_0]$ has a unique counitary bialgebra structure determined by
$\Delta|_{A\ot A}=\Delta_A$, $\Delta(x_0)=x_0\ot x_0$, $\varepsilon|_A=\varepsilon_A$ and
$\varepsilon(x_0)=1$. One can easily see that $\varepsilon(x_0D-1)=0=\varepsilon(Dx_0-1)$, so $J$
is included in the kernel of the counit. In particular $(A*\Bbbk[x_0])/J$ is a non-zero $\Bbbk$-algebra.
 \end{remark}

\begin{ex}
Let $G=F_2=F(x,y)$ the free group on two elements $x$ and $y$, $H=\Bbbk[G]$ the group algebra,
and $A\subset H$ the $\Bbbk$ subalgebra generated by $x$, $x^{-1}$ and $y$. Taking $D=y$, the
inclusion $A\to H$ has the universal property of the localization of $A$ in $D$. 
\end{ex}

In the above example we see that the localization is not necesarily a ``calculus of fractions",
in the sense that not every element in $H$ is of the form $y^{-n}a$ 
or $ay^{-n}$ for $a\in A$; that is, $y$ do not satisfy the Ore condition in 
$\Bbbk\langle x^{\pm 1},y\rangle\subset \Bbbk[F_2]$. Nevertheless,
our situation is much simpler:
\begin{rem}
Due tu Hayashi's result 
(see Lemma \ref{lemaHayashi}), for a coquasitriangular bialgebra $A$ and a (non-zero) 
group-like element $D$, there exists an automorphism $\J_D:A\to A$ such that
$Da=\J_D(a)D$ for all $a\in A$. So, 
the multiplicative set $\{D^{n}\}_{n\in \N_{0}}$ satisfies the Ore condition 
and
every element in $A[D^{-1}]$ can be written
as $D^{-n}a$ (or $aD^{-n}$) for some $a\in A$, $n\in \N_{0}$. In particular, the localization
$A[D^{-1}]$ as defined above coincides with the Ore-localization corresponding to the multiplicative set 
$\{D^{n}\}_{n\in \N_{0}}$. 
In particular, $A[D^{-1}]$ is a coquasitriangular bialgebra.

\end{rem}

We introduce now the localization of $A(c)$ in the quantum determinant.

\begin{defi}\label{def:H(c)}
Let $H(c)$ be the $\Bbbk$-algebra 
generated by the elements $\{t_{i}^{j}\}_{i,j}$ and $D^{-1}$ satisfying the relations \eqref{eq:FRT}
and 
\begin{equation}\label{eq:D}
DD^{-1}= 1= D^{-1}D. 
\end{equation}
 \end{defi}
\noindent 
It easy to see that $H(c)$ is indeed a localization of $A(c)$ in $D$. 
For this reason, we write indistinctly
$H(c)= A[D^{-1}]$; the canonical map is denoted by $\iota: A(c)\to H(c)$. See Section
 \ref{sec:examples} for examples.
The next result follows from \cite[Theorem 3.1]{H}.
 We give its proof for completeness. 
\begin{lem}\label{lem:H(c)bialgebra}
$H(c)$ is a coquasitriangular bialgebra.
\end{lem}

\begin{proof}
Let $A'$ be the algebra generated by the same elements but satisfying 
only \eqref{eq:FRT} and 
$$t_{i}^{j}D^{-1} = D^{-1}\J(t_{i}^{j}) \qquad \text{ for all }1\leq i,j\leq n.$$ 
Then,  $H(c)=A'/J$,
where $J$ is the two-sided ideal generated by the relation \eqref{eq:D}. 
In particular, 
we have that $\iota: A\to H(c)$ factorizes through $A'$.
Note that since $\J$ is a bialgebra map, one has that 
$aD^{-1} = D^{-1}\J(a)$ for all $a \in A$.

By defining $\Delta(D^{-1}) = D^{-1}\ot D^{-1}$ and $\eps(D^{-1})=1$, we may 
endow $A'$ with a bialgebra 
structure:  
since  $\J$ is a bialgebra map, one has that 
\[
\Delta(aD^{-1}-D^{-1}\J(a))=a_{(1)}D^{-1}\ot a_{(2)}D^{-1}-D^{-1}\J(a_{(1)})\ot D^{-1}\J(a_{(2)})
\]
\[
=a_{(1)}D^{-1}\ot a_{(2)}D^{-1}-D^{-1}\J(a_{(1)})\ot a_{(2)}D^{-1}+D^{-1}\J(a_{(1)})\ot a_{(2)}D^{-1}
-D^{-1}\J(a_{(1)})\ot D^{-1}\J(a_{(2)})
\]
\[
=\big(a_{(1)}D^{-1}-D^{-1}\J(a_{(1)})\big)\ot a_{(2)}D^{-1} +D^{-1}\J(a_{(1)})\ot\big( a_{(2)}D^{-1}-D^{-1}\J(a_{(2)})\big),
\]
for all $a\in A$. So, $\Delta$ is well-defined on $A'$. Also $\eps$ is well-defined since,
by
the explicit description of $\J$ (see Lemma \ref{lemaHayashi}), we have
$\eps(aD^{-1}) = \eps(a) = \eps(D^{-1}\J(a))$ for all
$a\in A$.

 To show that $H(c)$ is a bialgebra, it is enough to show that 
$J$ is also a coideal. This follows by a direct computation since
both $D$ and $D^{-1}$ are group-like elements.
Finally, the coquasitriangular structure is defined extending the coquasitriangular 
structure on $A$ by 
$r(D^{-1}, a) = r^{-1}(D,a)$ and $r(a, D^{-1}) = r^{-1}(a,D)$ for all $a\in A$. It is well-defined 
thanks to (CQT1)-(CQT2).

\end{proof}

\begin{rem}
In the category ${}^{H(c)}\M$, the comodule $\B^{top}$ is invertible, that is, there exists an 
$H(c)$-comodule $M$ such that $\B^{top} \ot M\simeq \Bbbk \simeq M\ot \B^{top}$; in particular,
$\B^{top}$ and $M$ are 
one-dimensional.
Indeed, consider  the one-dimensional vector space $\Bbbk D^{-1}$ with generator
$D^{-1}$ and whose left $H(c)$-comodule structure is given by
$
\lambda(D^{-1})=D^{-1}\ot D^{-1}
$.
Since 
the (diagonal) coaction on $\Bbbk \b\ot \Bbbk D^{-1}$ is trivial, i.~e.~
\[
\lambda(\b\ot D^{-1})=DD^{-1}\ot (\b\ot D^{-1})=
1\ot (\b\ot D^{-1}),
\]
it turns out that
$\Bbbk \b \ot \Bbbk D^{-1}\cong \Bbbk $ as $H(c)$-comodules. 
Similarly $\Bbbk D^{-1}\ot \B^{top}\cong \Bbbk $.
\end{rem}

\begin{defi}\label{defdual} Let $V^*$ and ${}^*V$ be the  
$H(c)$-comodules given by
\[
V^* :=\B^{top-1}\ot \Bbbk D^{-1}
,\
 {}^*V :=\Bbbk D^{-1}\ot \B^{top-1}
\]
Using that the multiplication $m_\B$ of $\B$ gives  non-degenerate colinear maps
\[
V\ot \B^{top-1}\to \B^{top}=\Bbbk \b,
\qquad \B^{top-1} \ot V \to \B^{top}=\Bbbk \b,
\]
we may define evaluation maps $\ev_\ell: {}^*V \ot V\to \Bbbk $ and $\ev_r:V\ot V^* \to \Bbbk $
by the following compositions. For the right evaluation:
\[
\xymatrix@-2ex{
V\ot V^*\ar@{=}[d]\ar@/^/[rrd]^{\ev_r}&&\\
V\ot (\B^{top-1}\ot \Bbbk D^{-1})\ar[r]&\Bbbk \b \ot \Bbbk D^{-1}\ar[r]^{\cong}&\Bbbk \\
x\ot (w \ot D^{-1})\ar@{|->}[r] &xw\ot D^{-1}&\\
                                                        &\ar@{=}[u] \ev_ r(x,w\ot D^{-1})\b\ot D^{-1} \ar@{|->}[r]&
 \ev_r(x,w\ot D^{-1})}
\]
and similarly on the left:
\[
\xymatrix@-2ex{
  {}^*V\ot V \ar@{=}[d]\ar@/^/[rrd]^{\ev_\ell}&&\\
(\Bbbk D^{-1}\ot \B^{top-1})\ot V\ar[r]&\Bbbk D^{-1} \ot \Bbbk \b\ar[r]^{\cong}&\Bbbk \\
(D^{-1}\ot w) \ot x\ar@{|->}[r] &D^{-1}\ot wx&\\
 &\ar@{=}[u] \ev_\ell(D^{-1} \ot w,x)\b\ot D^{-1} \ar@{|->}[r]&
 \ev_\ell(D^{-1} \ot w,x)}
\]
Observe that everything depends on the choice of the ``volume element" $\b$.
\end{defi}

\begin{rem}
By (WGF4) one may define
two (possibly different) bases for $\B^{top-1}$ which are
right and left dual to a given basis $\{x_1,\dots,x_n\}$ of $V$, say $\{w_\ell^1,\dots,w_\ell^n\}$ and $\{w_r ^1,\dots,w_r ^n\}$, 
satisfying for all $1\leq i,j\leq n$ that
\[
x_iw_\ell^j=\delta_i^j\b,\hskip 1cm
w_r ^jx_i=\delta_i^j\b. 
\]
\end{rem}

For $V^* =\B^{top-1}\ot \Bbbk D^{-1}$
and  $ {}^*V =\Bbbk D^{-1}\ot \B^{top-1}$ as above, define left and right coevaluation maps 
$\coev_{\ell}:\Bbbk \to V\ot  {}^*V$ and $\coev_{r}:\Bbbk \to {}^*V \ot  V$ by
\[
\coev_\ell(1):=\sum_ix_i\ot  (D^{-1}\ot w_\ell^i),\qquad 
\coev_r(1):=\sum_i (D^{-1}\ot w_r^i) \ot x_i.
\]

By a direct computation we obtain the following:
\begin{lem}
$V^* $ and $ {}^*V $ are, respectively, right and left duals of $V$ in $\ ^{H(c)}\M$.\qed
\end{lem}

\begin{rem} Using similar arguments as before, one has that 
$(V^*)^*=  \Bbbk \b\ot V\ot \Bbbk D^{-1}$ and 
${}^*({}^*V)=    \Bbbk D^{-1} \ot V\ot \Bbbk \b $. In particular, 
$V^{**} \simeq V \simeq {}^{**}V$.
\end{rem}

The next theorem is our first main result. It states that $H(c)$ is indeed a Hopf algebra,
 provided the 
canonical map $\iota: A(c)\to H(c)$ is injective. This is the case when $D$ is not a zero divisor in $A$.

\begin{teo}\label{teomain}
If the canonical map $\iota: A(c)\to H(c)$ is injective
 then ${}^{H(c)}\M$ is rigid, tensorially generated by
$V$ and $\Bbbk D^{-1}$.
As a consequence, $H(c)$ is a coquasitriangular
 Hopf algebra. Moreover, the formula for the antipode is
given on generators by
$\SS(D^{-1})=D$, and for all $1\leq i,j\leq n$:
\[
\SS(t_i^j):=T_i^jD^ {-1}.
\]
\end{teo}

\begin{proof} 
Identify the elements of $A$ with their image in $H(c)$ 
under the canonical map. Let 
$M \in {}^{H(c)}\M$  
and fix a basis $\{m_1,\dots, m_k\}$ of $M$ as vector space.
Denote by $h_i ^k\in H(c)$ the elements such that the coaction on $m_{i}$ is given by
$
\lambda(m_i)=\sum_j h_i ^j\ot m_j
$.
If all $h_i ^j$ belong to the image of $A$ under the canonical map, then clearly $M$ is an $A$-comodule.
Since $D$ is normal,
 each $h_ i^j$ can be written as a polynomial in $D^{-1}$ with coefficients in 
(the image of) $A$, say:
\[
h_i^j=\sum_ {k=0}^{N_{ ij}}a^{ ij}_kD^{-k} \qquad \text{ with }a^{ij} \in A\text{ for all }1\leq i,j\leq n.
\]
Thus, for all $ N\geq N_{ ij}$ we have
$
h_i^jD^{N}\in A
$.
Taking $N=\max_{i,j}\{N_ {ij}\}$, we have that
\[
\wt M = M\ot \Bbbk \b^{\ot N}=  M\ot \overbrace{ \Bbbk \b\ot\cdots\ot \Bbbk \b}^{N-\text{times}}
\]
is an $A$-comodule: 
a basis is $\{m_ i\ot \b\ot \b\ot\cdots \ot \b\}_{1\leq i\leq n}$, and
the $A$-comodule structure is
\[
\lambda (m_ i\ot \b\ot \b\ot\cdots \ot \b)=
\sum_j h_ i^j D^N\ot 
m_ i\ot \b\ot \b\ot\cdots \ot \b \in A\ot \wt M
\]
Now using that $\Bbbk \b\ot \Bbbk D^{-1}\cong \Bbbk$ we have that 
$ \wt M\ot \overbrace{\Bbbk D^{-1} \cdots \ot \Bbbk D^{-1}}^{N-\text{times}}\cong M$.
That is, $M$ is isomorphic to a tensor product of an $A$-comodule and the 
$H(c)$-comodule $\Bbbk D^{ -1}$. Thus ${}^{H(c)}\M$ is tensorially generated by $A(c)$-comodules and $\Bbbk D^{-1}$.
Since ${}^{A(c)}\M$ is tensorially generated by $V$, the first assertion of the statement follows.

Finally, we prove
the formula for the antipode.
Since $D$ and $D^{-1}$ are group-like, 
we have 
 $\SS(D)=D^{-1}$ and $\SS(D^{-1})=D$. 
For the generators $t_i^j$
we proceed as follows: from Proposition \ref{propFila} we know that
$
\sum\limits_{k}t_i^kT^j_ k = \delta_ i^j D
$. 
So, in $H(c)$ it holds that
\[
\sum\limits_{k}t_i^kT^j_ kD^{-1} = \delta_ i^j \qquad \text{ for all }1\leq i,j\leq n.
\]
Now, since $H(c)$ is a Hopf algebra, we must have that 
$
\sum_\ell \SS(t_i^\ell)t_\ell^k=\varepsilon (t_i^k)=\delta_i^k$.
From one side one gets
\[
(*):=\sum_ k\sum_ \ell \SS(t_i^\ell)t_\ell^kT^j_ k D^{-1}
=\sum_{k}\delta_i ^k T^j_k D^{-1}
=T^j_i D^{-1},
\]
but, changing the order of the double sum, we obtain
\[
(*)=\sum_\ell\sum_k \SS(t_i^\ell)t_\ell^k T^j_k D^{-1}
=\sum_\ell \SS(t_i^\ell) \delta_\ell^j
= \SS(t_i^j).
\]
Consequently, $\SS(t_i^j)=T_i^jD^{-1}$ for all $1\leq i,j\leq n$.

\end{proof}

\begin{rem}
The antipode verifies both $\SS *\id=u\varepsilon$ and $\id*\SS=u\varepsilon$, so
we also have
\[
\sum\limits_{k}T_i^ k D^{-1}t_k^j= \delta_ i^j.
\]
In particular, if $D$ is central, in addition to
$\sum\limits_{k}t_ i^ k T_k^j= \delta_ i^jD $, one must
have
$
\sum\limits_{k}T_ i^ k t_k^j= \delta_ i^jD.
$
In the general case, it holds that 
\[
\sum\limits_{k}\J(T_i^ k)t_k^j= \delta_ i^j D,
\]
where $\J$ is as in Lemma \ref{lemaHayashi}.
It would be interesting to have a direct proof of this fact in $A(c)$.
 \end{rem}

 Usually, one does not know \textit{a priori} if $\iota: A\to H(c)$ is injective, and there 
are examples where this map actually have non-zero kernel. Also, it is difficult to
check by computer if the element $D$ is a zero divisor or not. 
We give below a "computer adapted version" of Theorem \ref{teomain},
without the assumption of the canonical map $\iota: A\to H(c)$ being injective.

\begin{teo}\label{teomain-new}
Assume the following equality holds in $A(c)$ for all $1\leq i,j\leq n$:
\begin{equation}\label{eq:teocompfirendly}
\sum\limits_{k=1}^{n}\J(T_i^k)t^j_ k = \delta_ i^j D.
\end{equation}
Then $H(c)$ is a coquasitriangular Hopf algebra and the formula for the antipode  on generators is given by
$\SS(D^{-1})=D$, and
$
\SS(t_i^j):=T_i^jD^ {-1}
$ for all $1\leq i,j\leq n$.
\end{teo}

\pf 
Define an algebra map $\varphi: \mathcal{H}(c)\to H(c)$ by
$\varphi(\frak{t}_i^j)= t_i^j$ and $\varphi(\frak{u}_i^j)= T_i^jD^{-1}$  for all $1\leq i,j\leq n$.
It is clear that the FRT relations \eqref{eq:FRT} and \eqref{eq:relationsPeter} are the same, also
\[\varphi\big(
\sum_k \frak{t}_i^k \frak{u}_k^j\big)=
\sum_k t_i^k T_k^jD^{-1}=\delta_i ^j DD^{-1}=\delta_i^j.
\]
For the remaining  relations,
notice that for $a\in H(c)$, we have
$D^{-1}\J(a)= aD^{-1}$
and so
\[\varphi\big(
\sum_k \frak{u}_i^k \frak{t}_k^j\big)=
\sum_k T_i^jD^{-1} t_k^j=
\sum_k D^{-1}\J(T_i^j) t_k^j=
D^{-1}\sum_k \J(T_i^j) t_k^j=
D^{-1}\delta_i ^j D=\delta_ i ^j
\]
This implies that $\varphi$ is a well-defined algebra map. Moreover, by \ref{not:Tij} it follows that 
$\varphi$ is indeed a bialgebra map.
In particular, $\B$ is also a weakly graded-Frobenius algebra for
$\mathcal{H}(c)$ and there exists a group-like element $\mathcal{D}$ on $\mathcal{H}(c)$
which is mapped to $D$. Since $\mathcal{H}(c)$ is a Hopf algebra, $\mathcal{D}^{-1}$ is a group-like element 
whose image $D^{-1}$ is contained in the image of $ \varphi$.
Consequently, $f$ is surjective and $H(c)$ is a Hopf algebra. By the universal property of $\mathcal{H}(c)$, 
it follows that $H(c)$ is also coquasitriangular.

The formula for the antipode follows the same lines as in the proof of Theorem \ref{teomain}. 
\epf

We end this section with two
corollaries.  The first one states that actually
 both Hopf algebra $H(c)$ and $\mathcal{H}(c)$ coincide. The second one is a resume that
 stresses the results for Nichols algebras.

\begin{corollary}\label{cor:H=HPeter}
Assume that \eqref{eq:teocompfirendly} holds in $A(c)$ for all $1\leq i,j\leq n$. Then
 $H(c)$ and $ \mathcal{H}(c)$ are isomorphic as Hopf algebras.
\end{corollary}

\begin{proof}
By the proof of Theorem \ref{teomain-new}, we know that there is a surjective 
Hopf algebra map $\varphi: \mathcal{H}(c) \to H(c)$ such that 
$\varphi(\mathcal{D})= D$,  $\varphi(\frak{t}_{i}^{j}) = t_{i}^{j}$ and $\varphi(\frak{u}_i^j)= T_i^jD^{-1}$ for all 
$1\leq i,j \leq n$, where $\mathcal{D}$ is the quantum determinant of $\mathcal{H}(c)$ associated with $\B$.
To prove that both algebras coincide, we show that $\varphi$ is bijective. Define the algebra map ${f}:H(c) \to \mathcal{H}(c)$ by 
${f}(t_{i}^{j}) = \frak{t}_{i}^{j}$ and ${f}(D^{-1}) = \mathcal{D}^{-1}$. Clearly, it is a
well-defined bialgebra map, which satisfies that $\frak{u}_{i}^{j}= f(T_{i}^{j}D^{-1})$.
Indeed,
define the matrices $\bt, T,\J( T), \mathbf{\frak{t}}$ and $\mathbf{\frak{u}}$ by
$(\bt)_{ij}=t_i^j$, $(T)_{ij}=T_i^j$, $(\J (T))_{ij}=\J(T_i^j)$,
$(\frak{t})_{ij}=\frak{t}_i^j$ and $(\frak{u})_{ij}=\frak{u}_i^j$. In particular, 
by Proposition \ref{propFila} and our assumptions, we know that 
\[
f(\bt) = \frak{t},\qquad \frak{u}\cdot \frak{t} = \id = \frak{t}\cdot \frak{u}, \qquad \bt\cdot T=D I \qquad\text{ and }\qquad  \J( T)\cdot \bt= D I.
\]
Thus, 
\[
f(T)=(\frak{u}\cdot \frak{t})\cdot  f(T)=\frak{u}\cdot(f(\bt)\cdot f(T))=\frak{u}\cdot(f(\bt\cdot T)) = \frak{u} D.\]
Namely, $f(T_i^j)= \frak{u}_i^jD$ and so $\frak{u}_i^j=f(T_i^jD^{-1})$. Hence,
we conclude that the algebra map $f:H(c) \to \mathcal{H}(c)$
is surjective.
Since by definition we have that $f\circ \varphi=\id$ and $\varphi\circ f=\id$, the claim is proved.
\end{proof}

\begin{corollary}
Let  $(V,c)$ be a rigid finite-dimensional braided vector space such that the associated Nichols algebra 
$\B(V,c)$ is finite-dimensional. If the canonical map $\iota: A(c)\to H(c)$ is injective 
or equation \eqref{eq:teocompfirendly} is satisfied, then 
$\B(V,c)$ is a braided Hopf algebra in ${}^{H(c)}\M$.
\end{corollary}

\pf Follows from Proposition \ref{prop:B-A-comodulo}, the fact that $\iota: A(c) \to H(c)$ is
a bialgebra map and that $\B(V,c)$ is a braided Hopf algebra in ${}^{A(c)}\M$.
\epf

\begin{question} If the canonical map $\iota: A\to H(c)$ is injective, 
one can easily see that the hypothesis of Theorem
\ref{teomain-new} is superfluous. We do not know if it is actually
superfluous in general, or at least 
superfluous in some situation, \textit{e.g.} noetherianity.
\end{question}

\subsection{Formulas for the quantum determinant for set-theoretical solutions}
In this subsection, we give an explicit formula for the quantum determinant in the case 
where the braided vector space is given
by a linearization of a set-theoretical solution of the braid equation. 
Using the coquasitriangular structure, we also give the commuting relations between the 
quantum determinant and the generators of the bialgebra.

Let $(X,c)$ be a set-theoretical solution of the braid equation and write $$c(i,j) = (g_{i}(j),f_{j}(i))\qquad \text{ for all }i,j\in X,$$ 
where $f,g: X\to \Fun(X,X)$. We say that 
the solution $(X,c)$ is \textit{non-degenerate} if the images of $f$ and $g$ are bijections.
Let $V=\Bbbk X$ be the vector space linearly spanned by $X$ and consider the map
on $V$, also written as $c$, obtained by linearizing $c$. Then $(V,c)$ is a braided vector space. Set $\{x_{i}\}_{i\in X}$ for the linear basis of $V$. A 
map $\mathbf{q}:X \times X \to \Bbbk$ is called a \textit{cocycle} if the map 
$c^{\mathbf{q}}: V\ot V \to V\ot V
$ given by $$c^{\mathbf{q}}(x_{i}\ot x_{j})=q_{ij}x_{g_{i}(j)}\ot x_{f_{j}(i)} \qquad \text{ for all }i,j\in X,$$ 
where $q_{ij}=q(i,j)$, is a solution of the braid equation. It turns out that the braiding $c^{\mathbf{q}}$ is rigid  if and only if $c$ is non-degenerate, see 
\cite[Lemma 5.7]{AG2}.

Assume $|X|=n$ and let  $(c_{ij}^{k\ell})_{i,j,k,\ell\in X}$ be the $(n^{2}\times n^{2})$-matrix given by 
$c^{\mathbf{q}}(x_i\ot x_j)=\sum_{k,\ell}c_{ij}^{k\ell}x_k\ot x_\ell$. By the formula above, it follows that 
$$c_{ij}^{k\ell} = q_{ij}\delta_{k,g_{i}(j)}\delta_{\ell,f_{j}(i)}.$$
Let $\B(V,c^{\mathbf{q}})$ be the Nichols algebra associated with the rigid braided vector space $(V,c^{\mathbf{q}})$.
If $\dim \B(V,c^{\mathbf{q}})$ is finite, then $\B(V,c^{\mathbf{q}}) =\B= \bigoplus_{i=0}^{N} \B^{i}$ with 
$\dim \B^{N} = 1$. Assume further that $\B^{N} = \Bbbk \b$;  with
$\b$ a "volume element". Then for any element
$x_{i_{1}}\cdots x_{i_{N}} \in \B^{N}$ with $x_{i_{s}} \in V$, there 
exists $\alpha_{i_{1},\ldots , i_{N}} \in \Bbbk$ such that 
$x_{i_{1}}\cdots x_{i_{N}} =  \alpha_{i_{1},\ldots , i_{N}}\b$.

\begin{prop}\label{prop:formdet}
Assume $\b=x_{j_{1}}\cdots x_{j_{N}}$ with $x_{j_{s}} \in V$ for all 
$1\leq s\leq N$. Then 
$$D = \sum_{1\leq  {i_1,\dots,i_N}\leq n} \alpha_{i_{1},\ldots , i_{N}} t_{j_{1}}^{i_{1}}t_{j_{2}}^{i_{2}}\cdots t_{j_{N}}^{i_{N}}.$$
\end{prop}

\begin{proof}
 Since $\lambda(x_{i}) = \sum_{i=1}^{n} t_{i}^{j}\ot x_{j}$ for all $1\leq j\leq n$, we have that 
 \begin{align*}
\lambda(\b) = \lambda(x_{j_{1}}\cdots x_{j_{N}}) & = \sum_{1\leq i_1,\dots,i_N\leq n} t_{j_{1}}^{i_{1}}t_{j_{2}}^{i_{2}}\cdots 
 t_{j_{N}}^{i_{N}}\ot x_{i_{1}}x_{i_{2}}\cdots x_{i_{N}}\\ 
& = \sum_{1\leq i_1,\dots,i_N\leq n}
\alpha_{i_{1},\ldots , i_{N}}t_{j_{1}}^{i_{1}}t_{j_{2}}^{i_{2}}\cdots 
 t_{j_{N}}^{i_{N}}\ot \b \\
 \end{align*}
Since $\lambda(\b) = D\ot \b$, the assertion follows.
\end{proof}

Next we give some formulas concerning the quantum determinant
 and the coquasitriangular structure.

\begin{prop} Let $1\leq a,b \leq n$ and denote recursively 
$a_{1} = f_{j_{1}}(a)$ and $a_{k} = f_{j_{k}}(a_{k-1})$. Then
$$r(D,t_{a}^{b}) = \delta_{b,a_{N}} \alpha_{g_{a}(j_{1}),g_{a_{1}}(j_{2}))\ldots , g_{a_{N-1}}(j_{N})} 
q_{j_{1},a}q_{j_{2}, f_{a}(j_{1})}\cdots q_{j_{N},a_{N-1}} .$$
\end{prop}

\begin{proof} 
 From Proposition \ref{prop:formdet} we get
 \begin{align*}
  r(D,t_{a}^{b}) & = \sum_{1\leq i_{k}\leq n} \alpha_{i_{1},\ldots , i_{N}}
  r(t_{j_{1}}^{i_{1}}t_{j_{2}}^{i_{2}}\cdots 
 t_{j_{N}}^{i_{N}}, t_{a}^{b})\\
& = \sum_{1\leq i_{k}\leq n}\sum_{1\leq \ell_{k} \leq n} \alpha_{i_{1},\ldots , i_{N}}
  r(t_{j_{1}}^{i_{1}}, t_{a}^{\ell_{1}}) r(t_{j_{2}}^{i_{2}}, t_{\ell_{1}}^{\ell_{2}})\cdots 
 r(t_{j_{N}}^{i_{N}}, t_{\ell_{N-1}}^{b}) \\
 & = \sum_{1\leq i_{k}\leq n}\sum_{1\leq \ell_{k} \leq n} \alpha_{i_{1},\ldots , i_{N}}
  c_{a,j_{1}}^{i_{1},\ell_{1}} c_{\ell_{1}j_{2}}^{i_{2},\ell_{2}}\cdots 
 c_{\ell_{N-1}j_{N}}^{i_{N}b} \\
 & = \sum_{1\leq i_{k}\leq n} \alpha_{i_{1},\ldots , i_{N}}\sum_{1\leq \ell_{k} \leq n}
  q_{a,j_{1}}\delta_{i_{1},g_{a}(j_{1})}\delta_{\ell_{1},f_{j_{1}}(a)}
  q_{\ell_{1},j_{2}}\delta_{i_{2},g_{\ell_{1}}(j_{2})}\delta_{\ell_{2},f_{j_{2}}(\ell_{1})}
\cdots \\
& \qquad\cdots  q_{\ell_{N-1},j_{N}}\delta_{i_{N},g_{\ell_{N-1}}(j_{N})}\delta_{b,f_{j_{N}}(\ell_{N-1})}\\
& = \alpha_{g_{a}(j_{1}),g_{a_{1}}(j_{2}))\ldots , g_{a_{N-1}}(j_{N})} 
q_{j_{1},a}q_{j_{2}, f_{a}(j_{1})}\cdots q_{j_{N},a_{N-1}} \delta_{b,a_{N}},
 \end{align*}
and the claim follows.
\end{proof}

\begin{remark}\label{rmk:Rqconst}
Assume $\mathbf{q}$ is a constant cocycle with $q_{ij}=q \in \Bbbk^{\times}$ for
 all $1\leq i,j \leq n$. Then 
 $$r(D,t_{a}^{b}) = \delta_{b,a_{N}} \alpha_{g_{a}(j_{1}),g_{a_{1}}(j_{2}))\ldots , g_{a_{N-1}}(j_{N})} 
q^{N}.
$$
\end{remark}

\begin{corollary}\label{coronormal}
 Let $1\leq a,b\leq n$. As before, denote $a_{k} = f_{j_{k}}(a_{k-1})$ with $a_{(0)} = a$, and set
 $b_{k} = f_{j_{k}}(b_{k-1})$ with $b_{N}=b$.
Assume $\mathbf{q}$ is a constant cocycle
 with $q_{ij}=q \in \Bbbk^{\times}$ for all $1\leq i,j \leq n$.
 Then 
 $$\alpha_{g_{a}(j_{1}),g_{a_{1}}(j_{2}))\ldots , g_{a_{N-1}}(j_{N})} 
D  t_{a_{N}}^{b} =  
\alpha_{g_{b_{0}}(j_{1}),g_{b_{1}}(j_{2}))\ldots , g_{b_{N-1}}(j_{N})} 
 t_{a}^{b_{0}} D.$$
In particular, if $\alpha_{g_{a}(j_{1}),g_{a_{1}}(j_{2}))\ldots , g_{a_{N-1}}(j_{N})} \neq 0$, we have that 
$$\J(t^{a}_{a_{N}})= \alpha_{g_{a}(j_{1}),g_{a_{1}}(j_{2}))\ldots , g_{a_{N-1}}(j_{N})} ^{-1}
\alpha_{g_{b_{0}}(j_{1}),g_{b_{1}}(j_{2}))\ldots , g_{b_{N-1}}(j_{N})} 
 t_{a}^{b_{0}}. $$
\end{corollary}

\begin{proof}
Follows directly from Remark \ref{rmk:D-coquasi} (4).
\end{proof}

\begin{rem}
Let $(X,\trid)$ be a rack and consider the associated set-theoretical solution of the braid equation $c(x,y) = (x\trid y, x)$ for all
$x,y\in X$. Set $V=\Bbbk X$ and consider the covering group 
\[
G:=G(X,c)=\langle g_{x}: x\in X\rangle/(g_{x}g_{y}=g_{x\trid y}g_{x}).
\]
Let $\mathbf{q}$ be a $2$-cocycle and write $c^{\bq}\in \Aut(V\ot V)$ for the braiding on $V$. Then $V$ is a left $\Bbbk G$-comodule 
with structure map $\delta_{V}: V\to \Bbbk G \ot V$ given by 
$\delta_{V} (x) = g_{x}\ot x$ for all $x\in X$. Moreover, $c^{\bq}$ is a $\Bbbk G$-colinear map and,  
by the universal property of $A(c^{\bq})$, there exists a bialgebra map $f:A(c^{\bq}) \to \Bbbk G$ such that 
$(f\ot \id)\lambda_{V} = \delta_{V}$. In particular, we have that 
\[
g_{x}\ot x = \delta_{V}(x) = (f\ot \id)\lambda_{V}(x) 
= (f\ot \id)\left(\sum_{y\in X} t_{x}^{y} \ot y\right) 
=\sum_{y\in X} f(t_{x}^{y}) \ot y,
\]
which implies that $f(t_{x}^{y}) = \delta_{x,y} g_{x}$ for all $x,y \in X$. Clearly, $f$ is a well-defined 
bialgebra map. As the
FRT relations of $A(c^{\bq})$ are given 
by 
$ q_{ij}\, t_{i\rhd j}^{\ell \rhd k}t_{i}^{\ell} =  q_{\ell,k}\, t_{i}^{\ell}t_{j}^{k}$ for all $i,j,k,\ell \in X$, 
one might view $\Bbbk G$, as the universal Hopf algebra of 
the algebra generated by the elements $\{g_{x}=t_{x}^{x}\}_{x\in X}$ satisfying the relations
$$  t_{x\trid y}^{x\trid y}t_{x}^{x} =  t_{x}^{x}t_{y}^{y}\qquad \text{ for all }x,y \in X.
$$
It is a bialgebra with the coalgebra structure determined by $t_{x}^{x}$ being group-like for all $x\in X$.

Suppose that $\B(V,c^{\bq}) $ is finite-dimensional and denote by $D \in A(c^{\bq})$ the associated 
quantum determinant. As $D$ satisfies \eqref{eq:A(c)-D}, we must have that 
$$
\delta_{i,j} f(D) = 
\sum_{k=1}^n f(t_i^k)f(T^j_k) = f(t_{i}^{i}) f(T^{j}_{i})= g_{i} f(T^{j}_{i}) \qquad \text{ for all } i,j \in X.
$$
Since $f(D),g_{i}$ are group-like elements in $\Bbbk G$, we have that $f(T^{j}_{i}) = 0$ if $i\neq j$
and $f(T^{i}_{i}) = g_{i}^{-1} f(D) \in G$ for all $ i\in X$.
From the (quantum) geometrical point of view, one may consider the universal map 
$\tilde{f}: H(c^{\bq}) \to \Bbbk G$ as the inclusion of the 
(non-commutative) subgroup of diagonal matrices into the quantum group associated with $H(c^{\bq})$. 
Compare with examples in Subsection \ref{subsec:non-injective}.
\end{rem}


\section{Examples}\label{sec:examples}
In this section we provide examples and formulas 
given by our main results.
Recall that given a set-theoretical solution of the braid equation $s:X\times X\to X\times X$, 
$s(x,y) = (g_x(y),f_y(x))$ for all $x,y\in X$,
there is a
so-called derived solution that is of rack type 
$\tau_{s}:X\times X\to X\times X$, 
that is, it is
of the form
$\tau_s(x,y)=(y,x\t_s y)$ for all $x,y\in X$ (see for instance \cite[Proposition 5.4]{AG2}
and references therein).
Now for every $n$, one can let act the braided group $\mathbb{B}_n$
using $s$ or $\tau_s$. The derived solution has the remarkable property that
there exists a bijection in $X^n$ intertwining these two possible actions. As a consequence,
for any non-zero constant cocycle $q$,
the dimensions of the homogeneous components Nichols algebras attached to $(X,qs)$
and $(X,q\tau_s)$ are the same. A very particular case is when $s^2=\id$; we also have
$\tau_s^2=\id$ and so necesarily $\tau_s$ is the flip. Hence, for {\em involutive set
 theoretical solutions}
$(X,s)$, one always have that the dimension of $\B(X,-s)$ is finite, 
and equal to the dimension of the exterior algebra. Nevertheless, the FRT-construction
on the braiding $s$ is far from being trivial:
Example \ref{ex:non-diag2x2} is the smallest example of a non-trivial
set-theoretical involutive solution, and Example \ref{ex:involutive-non-central} is also comming 
from an involutive set-theoretical solution whose quantum determinant is not central.

\subsection{A non-diagonal type  $2\times 2$ example}\label{ex:non-diag2x2}
Let $X=\{1,2\}$, and consider the set-theoretical solution of the braid equation given by 
$s:X\times X\to X\times X$
given by
\[
s(1,2)=(1,2),\qquad  s(2,1)=(2,1),\qquad
s(1,1)=(2,2),\qquad  s(2,2)=(1,1).
\]
Write 
$(t_ i^j)_{i,j}=
\left(
\begin{smallmatrix}
a&b\\
c&d\\
\end{smallmatrix}\right)$.
Then, the FRT relations \eqref{eq:FRT} are
\[
a^2 = d^2 ,\qquad
 ab = cd ,\qquad  ba = dc ,\qquad
 ac = bd ,\qquad ca = db,\qquad  
 b^2 = c^2.\]
The Nichols algebra $\B=\B(V,c)$ associated with $V=\Bbbk x\oplus \Bbbk y$ and the linearization of $c=-s$ is 
the $\Bbbk$-algebra generated 
by $x$, $y$ with relations
\[
x^2+y^2=0,\qquad 
2xy=0=2yx.
\]
If char$(\Bbbk)\neq 2$, then $\dim \B$ is finite and $\B$ has 
a basis $\{1,x,y,x^2\}$. The volume element $\frak{b}$ can be taken
to be $\frak{b}=x^2$. Since
$
\lambda (x)= a\ot x+b\ot y$, we have that 
\[
\lambda (\frak{b})=\lambda (x^2)=
( a\ot x+b\ot y)^2=a^2\ot x^2+ab\ot xy + ba\ot yx+b^2\ot y^2
=(a^2-b^2)\ot x^2.
\]
Thus, we obtain that $D:=a^2-b^2$ is a group-like element. One can check by hand that it is 
central (hence $\J=\id$) and   $\{\om_1=x,\om_2=-y\}$
is a "dual basis" with respect to the volume element $\frak b$. We compute the coaction to get the values of the $T_i^j$:
\[
\lambda (\om_1)=\lambda (x)=a\ot x+b\ot y=a\ot \om_1-b\ot \om_2,\]
\[
\lambda (\om_2)=\lambda (-y)=-c\ot x-d\ot y=-c\ot \om_1+d\ot \om_2.\]
Actually, it is quite difficult to check whether $D$ is a zero divisor or not. However,
to check the condition of Theorem \ref{teomain-new} is a 
very easy task (one can check it directly by hand or use GAP)
and
conclude that $H(s)=A[D^{-1}]=: \GL(X,-s)$ is a Hopf algebra.
The antipode is given by
\[
\SS(a)=aD^{-1},\
\SS(b)=-cD^{-1},\
\SS(c)=-bD^{-1},\
\SS(d)=dD^{-1}.
\]
Since the quantum determinant $D$ is central, we may also consider the Hopf algebra 
$\SL(X,-s) $ given by $A(s)/(D-1)$. It is the algebra presented by 
$$ \SL(X,-s) = \Bbbk\langle a,b,c,d: a^2 = d^2 ,
 ab = cd ,  ba = dc ,
 ac = bd , ca = db, 
 b^2 = c^2,\ a^{2}-b^{2}=1 \rangle. 
$$

\subsection{Involutive and non-central example}\label{ex:involutive-non-central}
For $X=\{1,2,3\}$, consider the set-theoretical solution of the braid equation given by
$s(i,i)=(i,i)$ for  $i=1,2,3$, 
$s(i,j)=(j,i)$  for  $i,j=2,3$,  and 
\[
s(1,2)=(3,1),\qquad s(1,3)=(2,1),\qquad 
s(2,1)=(1,3),\qquad s(3,1)=(1,2).
\]
Clearly, this solution is involutive.
For
$(t_ i^j)_{i,j}=
\left(
\begin{smallmatrix}
a&b&c\\
d&e&f\\
g&h&i
\end{smallmatrix}\right)$,
 the FRT relations 
 reads 
\[ 
c^2 = b^2 ,\
 g^2 = d^2 ,\
 h^2 = f^2 ,\
 i^2 = e^2 ,\]\[
 ba = ac ,\
 ca = ab ,\
 da = ag ,\
 db = cg ,\
 dc = bg ,
 ea = ai ,\]\[
 eb = ci ,\
 ec = bi ,\
 eg = di ,\
 fa = ah ,\
 fb = ch ,
 fc = bh ,\]\[
 fg = dh ,\
 fi = eh ,\
 ga = ad ,\
 gb = cd ,\
 gc = bd ,\
 gh = fd ,
 gi = ed ,\]\[
 ha = af ,\
 hb = cf ,\
 hc = bf ,\
 hd = gf ,\
 hg = df ,\
 hi = ef ,
 ia = ae ,\]\[
 ib = ce ,\
 ic = be ,\
 id = ge ,\
 if = he ,\
 ig = de ,
 ih = fe. \]
 Let $V=\Bbbk X$ and write $s$ also for the braiding given by the linearization of $s$. 
As the set-theoretical 
solution  is involutive, the Nichols algebra $\B(V,-s)$ is finite-dimensional, 
its maximal degree is $3$.
Our construction gives the quantum determinant
\[
\fbox{$
D= ae^2 - af^2 + bdf - bed - cde + cfd 
$}\]
It is group-like, normal but {\em not central}:
$D$ commutes with $a$ but  
$bD=-Dc$, $cD=-Db$, $dD=-Dg$, $gD=-Dd$, $De=iD$, $Di=eD$, $Df=hD$, $fD=Dh$.
On the other hand, these non-commutation relations give us the formula
for the automorphism $\J$:
\[
\J(a)=a,\
\J(b)=-c,\
\J(d)=-g,\
\J(e)=i,\
\J(f)=h.
\]
Also, it holds that $\J^2=\id$. One can check directly that the hypothesis of Theorem  \ref{teomain-new} holds
and conclude that $H(s)=: \GL(X,-s)$ is a Hopf algebra. We also have the explicit 
formula for the antipode: 
\[
\big(\SS(t_i^j)\big)_{ij}=
\left(
\begin{array}{ccc}
- fh + ei 
& - ce + bf 
& ch - bi 
\\
 - fg + dh 
& - cd + ae 
& cg - ah 
\\ 
eg - di 
& bd - af 
& - bg + ai 
\end{array}
\right)
D^{-1}.
\]
\begin{rem}
The relations defining the FRT-construction in this example are not very enlightning,
however,  we exhibit them  the following reasons: first,
to stress the fact that our constructions are very explicit; second, to show that  
our methods apply to every braiding comming from a set theoretical involutive solution, as 
it has a finite-dimensional 
Nichols algebra attached to it.
Also, even for very elementary solutions (e.g. a braiding coming from a set 
theoretical involutive solution on a set with 3 elements!), 
the Hopf algebras that arise in this way are non-trivial, since the quantum
 determinants in these cases are not necessarily central. And third, the number of set-theoretical involutive
 solutions on a finite set $X$ grows really fast with respect to the cardinal of $X$, so, one has
a big number of exotic examples.
\end{rem}

\subsection{Fomin-Kirillov algebras}\label{subsec:FK}
Before introducing quantum determinants for Fomin-Kirillov algebras, we first apply our construction to 
solutions of the braid equation given by a rack and a cocycle.

A \textit{rack} is a pair $(X,\rhd)$ where $X$ is a non-empty set and 
$\rhd: X \times X \to X$ is a map such that $x\rhd (y\rhd z) = (x\rhd y)\rhd (x\rhd z)$
 and $x\rhd \underline{\quad}$ is
bijective for every $x,y,z \in X$. Every rack gives a set-theoretical solution of the braid equation by setting
\[
c(i,j) = (i\rhd j, i)\qquad \text{ for all }i,j\in X.
\]
A rack 2-cocycle $\bq:X\times X \to \Bbbk^{\times}$, $(i,i)\mapsto q_{i,j}$ is a function 
such that 
\[
 q_{i, j\rhd k} q_{j,k} = q_{i\rhd j, i\rhd k}q_{i,k}\qquad \text{for all }	i,j,k \in X. 
\]
Let $(X,\rhd)$ be a rack with $|X|=n$ and let $\bq:X\times X \to \Bbbk^{\times}$ be a 
cocycle. Then, 
one may define a braiding on the vector space  
$ V=\Bbbk X$ by
$$ c^{\bq}(x_{i}\ot x_{j}) = q_{ij}x_{i\rhd j} \ot x_{i}\qquad \text{ for all }i,j\in X.$$
If we write $c(x_{i}\ot x_{j}) = \sum_{k,\ell=1}^{n} c_{i,j}^{k,\ell}x_{k} \ot x_{\ell}$, then we 
have that $$c_{i,j}^{k,\ell} = q_{ij}\delta_{i\rhd j,k}\delta_{ i,\ell}\qquad \text{ for all }i,j\in X. $$
In particular, the FRT-relations defining $A(c^{\bq})$ have the following form
\[
q_{ij}\, t_{i\rhd j}^{k}t_{i}^{\ell} =  q_{\ell,\ell\rhd^{-1}k}\, t_{i}^{\ell}t_{j}^{\ell\rhd^{-1} k}
\]
Moreover, if we replace $\ell \rhd^{-1}k$ by $k$ we get 
\[
\fbox{\fbox{$ q_{ij}\, t_{i\rhd j}^{\ell \rhd k}t_{i}^{\ell} =  q_{\ell,k}\, t_{i}^{\ell}t_{j}^{k}
$}}\]

Let $n \in \{3,4,5\}$. The Fomin-Kirillov algebras $\mathcal{E}_{n}$ arise as Nichols algebras 
when one consider the solution of the braid equation associated with
the racks given by the conjugacy classes of transpositions in $\mathbb{S}_n$ and a constant
 cocycle, see 
\cite{MS}, \cite{AG2}, \cite{GGI} for more details. 
We describe explicitly the case 
when $n=3$ and $q_{ij} = -1$ for all $i,j\in X$.

Let $X= \mO_2^{\s_{3}}$ be the rack of transpositions in $\s_{3} 
$ and 
consider the constant cocycle $q_{ij} = -1$.
Let $V=\Bbbk \mO_2^{\s_{3}}$ be the braided vector space associated with them and take the basis 
 $x_{1}= x_{(12)}, x_{2}=x_{(13)}$ and $x_{3}=x_{(23)}$ on $V$. Then  
$ c(x_{i}\ot x_{j}) = -x_{i\rhd j} \ot x_{i}$ for all $1\leq i,j\leq 3$.
In this case, 
$A(c^{\bq})$ is generated by the elements $\{t_{i}^{j}\}_{1\leq i,j\leq 3}$ satisfying the 
relations  
$ t_{i\rhd j}^{k}t_{i}^{\ell} = t_{i}^{\ell}t_{j}^{\ell\rhd k}$ for all $1\leq i,j,k,\ell\leq 3$.
Because the cocycle is constant
 we 
may write
\begin{equation}\label{eq:commrelA}
\fbox{\fbox{$ t_{i\rhd j}^{\ell\rhd k}t_{i}^{\ell} = t_{i}^{\ell}t_{j}^{k}
$}}
 \end{equation}
The Nichols algebra $\B(\mO_2^{\s_{3}},-1)$ associated with this rack and cocycle is 
finite-dimensional and it is generated by the elements $x_{1},x_{2},x_{3}$ satisfying the
relations 
\begin{equation}\label{eq:rel-fk}
x_{i}^2=0,\qquad x_{1}x_{2}+x_{2} x_{3}+x_{3} x_{1} = 0,\qquad
 x_{1}x_{3}+x_{3} x_{2}+x_{2} x_{1}=0, 
\end{equation}
for all $1\leq i \leq 3$. 
It has dimension 12 and its volume element is in degree $N=4$.
 In our case, 
we may take  $\B^{4}=\Bbbk x_{1}x_{2}x_{3}x_{2}$ and the volume element $\b=x_{1}x_{2}x_{3}x_{2}$. 
In particular, by Remark \ref{rmk:Rqconst}, we have that 
$r(D,t_{i}^{j}) = \delta_{i,j}$ for all $1\leq i,j \leq 3$.
Using  
relations \eqref{eq:rel-fk} one may prove that
 \begin{align*}
\label{eq:topdegreefk}
  &x_{1}x_{2}x_{1}x_{2} = 0, & 
  &x_{1}x_{2}x_{1}x_{3} = -\b,  &  
  &x_{1}x_{2}x_{3}x_{1} = 0,&
  &x_{1}x_{3}x_{1}x_{2} = -\b,\\
  &x_{1}x_{3}x_{1}x_{3} = 0, &   
  &x_{1}x_{3}x_{2}x_{1} = 0,  & 
  &x_{1}x_{3}x_{2}x_{3} = \b,&
  &x_{2}x_{1}x_{2}x_{1} = 0,\\ 
  &x_{2}x_{1}x_{2}x_{3} = -\b, &  
  &x_{2}x_{1}x_{3}x_{1} = \b,  &  
  &x_{2}x_{1}x_{3}x_{2} = 0,&
  &x_{2}x_{3}x_{1}x_{2} = 0,\\ 
  &x_{2}x_{3}x_{1}x_{3} = \b, &   
  &x_{2}x_{3}x_{2}x_{1} = -\b,  &  
  &x_{2}x_{3}x_{2}x_{3} = 0,&
  &x_{3}x_{1}x_{2}x_{1} = \b,\\ 
  &x_{3}x_{1}x_{2}x_{3} = 0, &   
  &x_{3}x_{1}x_{3}x_{1} = 0,  &  
  &x_{3}x_{1}x_{3}x_{2} = -\b,&
  &x_{3}x_{2}x_{1}x_{2} = \b,\\ 
  &x_{3}x_{2}x_{1}x_{3} = 0, &   
  &x_{3}x_{2}x_{3}x_{1} = -\b,  &  
  &x_{3}x_{2}x_{3}x_{2} = 0.&
   & 
\end{align*}
For example, since $x_{1}x_{2}+x_{2} x_{3}+x_{3} x_{1}=0$ and $x_{i}^{2}=0$,	 it follows that
$$x_{2}x_{1}x_{2}x_{1} =0,\qquad x_{1}x_{2}x_{3}x_{1}=0,\qquad x_{2}x_{1}x_{2}x_{1}=0,
\qquad x_{2}x_{3}x_{1}x_{2}=0.
$$
On the other hand, $x_{1}x_{2}x_{3}x_{2} + x_{1}x_{3}x_{1}x_{2}=0$, which implies 
that $ x_{1}x_{3}x_{1}x_{2}=-x_{1}x_{2}x_{3}x_{2} = -\b$. Moreover, for all 
$1\leq r\leq 3$ we have that 
$ x_{r\rhd 1}x_{r\rhd 2}x_{r\rhd 3}x_{r\rhd 2}=x_{1}x_{2}x_{3}x_{2}$.
Writing 
$
(t_i^j)_{i,j=1,2,3}=\left(\begin{smallmatrix}
a&b&c\\
d&e&f\\
g&h&i
\end{smallmatrix}
\right)$,
the relations in $A(c)$ read
\[
ba = ac = cb ,
bc = ab = ca ,
da = ag = gd ,
db = bi = id ,\]\[
dc = ch = hd ,
dg = ad = ga ,
dh = cd = hc ,
di = bd = ib ,\]\[
ea = ai = ie ,
eb = bh = he ,
ec = cg = ge ,
ed = df = fe ,\]\[
ef = de = fd ,
eg = ce = gc ,
eh = be = hb ,
ei = ae = ia ,\]\[
fa = ah = hf ,
fb = bg = gf ,
fc = ci = if ,
fg = bf = gb ,\]\[
fh = af = ha ,
fi = cf = ic ,
hg = gi = ih ,
hi = gh = ig .\]
Thus, the quantum determinant is given by  
\[
\fbox{$
\begin{array}{rcl}
D&=&a^2e^2 - abdf + a^2f^2 - abgi + b^2d^2 - abgi \\
&&- abdf + b^2f^2 - abdf + c^2d^2 - abgi + c^2e^2\\
&=&c^2e^2 + c^2d^2 + b^2f^2 + b^2d^2 - 3abgi - 3abdf + a^2f^2 + a^2e^2.
\end{array}$}\]
One can check explicitly by hand using the relations above (or using GAP and non-commutative Gr\"obner basis)
that $D$ is a central element; in particular, the hypothesis
of Theorem \ref{teomain-new} holds. Thus,
$H(c)=: \GL(\mO_2^{\s_{3}},-1)$ is a Hopf algebra.
The formula for the antipode follows from considering dual bases in the Nichols algebra
and finding the elements $T_i^j$; this can be done explicitly.
For example,
\[
\SS(a)=
(- fbi + fah - ech + eai )D^{-1}.
\]
As $D$ is a central group-like element, one may also define the Hopf algebra 
$\SL(\mO_2^{\s_{3}},-1) $ given by $A(c)/(D-1)$.

We end this example with a question suggested by the referee.
\begin{question}
The example above relies heavily on computations. Following our construction, it is
 possible to describe the 
quantum function algebras $H(c)= \GL(\mO_2^{\s_{n}},-1)$ for $n=4$ or $n=5$. 
However, our methods  need computer assistance  since the dimension
of the Nichols algebra for $n=4$ is 576 and its top degree is 12, while
for $n=5$ the dimension
of $\B$ is 8294400 and its top degree is 40.
It would be
 interesting to present
$H(c)= \GL(\mO_2^{\s_{n}},-1)$ in a more conceptual way, since this algebra would give some insight
on the Fomin-Kirillov algebras.
 \end{question}

\subsection{Quantum determinants for quantum planes}
In \cite{AJG}, the authors consider all solution of the QYBE in dimension $2$ and give several
examples of finite-dimensional Nichols algebras, arranged in families
$\fR_{0,1}$,
$\fR_{1,i}$ ($i=1,2,3,4$), and $\fR_{2,i}$ ($i=1,2,3$). They remark that, up to now,
the only case known
where one can find a quantum determinant and localize $A(c)$ to obtain a
Hopf algebra is $\fR_{2,1}$, due to a result of Takeuchi \cite{T1}.
As our method only has as hypothesis $\dim\B<\infty$, we can apply it in the other cases (and for certain parameters) 
to obtain quantum determinants. 

For example, let us consider the case $\fR_{2,2}$ (with $k^{2}=-1 $ and $pq=1$, 
according to the notation in \cite{AJG}). 
The braiding associated
with this two-dimensional vector space $V_{k,p,q}$  is  
$$ \big(c(x_{i}\ot x_{j})\big) _{1\leq i,j\leq 2} = \left(\begin{matrix}
                                           -x_{1}\ot x_{1} & kq\ x_{2}\ot x_{1} -2\ x_{1}\ot x_{2}\\
                                           kp\ x_{1} \ot x_{2} & -x_{2} \ot x_{2}
                                          \end{matrix}\right).$$
The corresponding Nichols algebra is presented as follows 
$$\B(V_{k,p,q}) = T(V_{k,p,q}) / (x_{1}x_{2} -kq\ x_{2}x_{1}, x_{1}^{2}, x_{2}^{2} ).
$$
A PBW-basis is given by  $\{1, x_{1}, x_{2}, x_{1}x_{2}\}$, $\dim \B(V_{k,p,q}) = 4$ 
and one may take the volume element 
$\b= x_{1}x_{2}$.
Write
$(t_i^j)_{i,j=1,2}=\left( \begin{smallmatrix}
a&b\\
c& d
 \end{smallmatrix}\right)$. Then
the FRT relations read
$$ ab=kp\ ba,\qquad ac = kq\ ca, \qquad bc=q^{2}\ cb,\qquad
ad-da=kp\ bc,\qquad  cd = kp\ dc,\qquad bd=kq\ db.
$$
The quantum determinant is given by  
$
D=ad - kp\ bc$. This element is {\em not} central, it verifies
 $aD=Da$, $dD=Dd$,  
but $Db=p^2bD$ and $Dc=q^2cD$. Hence
\[\J(a)=a,\qquad \J(d)=d,\qquad
\J(b)=p^2b,\qquad \J(c)=q^2c.\]
The matrices $T$ and $\J(T)$ are
\[
T=(T_i^j)=\left(
\begin{array}{cc}
d&kqb\\
-kpc&a
\end{array}
\right),\hskip 1cm 
\J (T)=(\J(T_i^j))=\left(
\begin{array}{cc}
d&kpb\\
-kqc&a
\end{array}
\right)\]
One can easily check that
$\frak{t}\cdot T=D\cdot \id=\J(T)\cdot \frak{t}$,
where
$\frak{t}=(t_i^j)=\left(
\begin{smallmatrix}
a&b\\
c&d\\
\end{smallmatrix}
\right)$. 

\subsection{A non-quadratic Nichols algebra \label{ejchino}}
Another feature of \cite{AJG} is the presentation by generators and relations of families of   
Nichols algebras having quadratic relations.
In \cite{xiong}, the author presents another example over a quantum plane considered in \textit{loc.~cit.~}, but 
with no quadratic relations. It is a finite-dimensional Nichols algebra with all relations of
order bigger than $2$. 
As example, we compute the quantum determinant for the Nichols algebra associated with the quantum plane $V_{4,1}$.

Let $\xi$ be a primitive 6-root of unity and write $\xi=-\om$, with $\om$ a primitive 3-root of 1.
Let  $\B$ be the $\Bbbk$-algebra generated by $x$ and $y$ with relations
\[
x^3= 0,\qquad
y^3-x^2y-yx^2+xyx= 0,\qquad
y^2x+xy^2-yxy= 0,\qquad
\xi x^2y+\xi^{5}yx^2+xyx= 0.
\]
The volume element  is $\b= x^2yxy^2$,
and the quantum determinant for
$(t_i^j)_{i,j}=
\left(
\begin{array}{cc}
a&b\\
c&d\\
\end{array}\right)$, is
 \[
D=(-\om +\om ^2)b^2dbdc +( -2\om -\om ^2)b^2dbcd + (-\om -2\om ^2)b^2dad^2
 + (\om -\om ^2)b^2cbd^2 + \om b^2cbc^2 \]\[
\qquad + \om ^2b^2cadc 
+( 2\om +\om ^2)
badbd^2 + \om badbc^2 - badacd  -\om bacbdc 
-\om bacbcd  
+(\om +2\om^2)abdbd^2\]\[
 -\om ^2abdadc - abdacd  -\om ^2abcbdc + \om ^2abcad^2 + a^2dbcd + a^2dad^2.
\]

\subsection{Two examples of the non-injectiveness of the canonical map $\iota: A(c)\to H(c)$}
\label{subsec:non-injective}
\subsubsection{Commutative example}\label{EjemploPeter}
We thank Peter Schauenburg for providing us this example, together with an argument. Nevertheless,
we exhibit a different argument to show that the canonical map is not injective.
Consider the vector space 
$V=\Bbbk x\oplus \Bbbk y$ with the following braiding
 \[
c(x\ot x)=x\ot x,\qquad
c(x\ot y)=y\ot x,\qquad
c(y\ot x)=x\ot y,\qquad
c(y\ot y)=2y\ot y.
\]
The FRT relations of $A(c)$ are
\[
ba - ab =0,\
 b^2 =0,\
 bd =0,\
 ca - ac =0,\
 cb - bc =0,\]\[
 c^2 =0,\
 cd =0,\
 da - ad= 0,\
 db - 2bd= 0,\
 dc - 2cd =0,
\]
which are equivalent to
\[ab=ba,\qquad ac=ca,\qquad ad=da,\qquad  bc=cb,\qquad
0=bd=db=cd=dc = b^2 =c^2.\]
Thus, $A(c) = A=k[a,b,c,d]/(b^2,c^2,bd,cd)$ is a commutative bialgebra. Assume $\Bbbk$ is algebraically closed of characteristic zero.
Then,
it is clear that $A$ cannot inject into a Hopf
algebra because it has nilpotent elements. Also, a direct proof 
(valid on any characteristic) in this particular 
case can be given as follows:
the bialgebra $A$ is a quotient of $\Bbbk[a,b,c,d]=\O(\Mq_2)$, thus
$D:=ad-bc$ is a group-like element in $A$. Notice that $\Lambda V$ is not
the Nichols algebra of $(V,qc)$ for any $q\in\Bbbk^{\times}$, because the former is
 finite-dimensional and 
the latter infinite-dimensional, but
nevertheless it is a weakly graded-Frobenius algebra for $A$.
Since
$b^2=0=bd$, it follows that  
$bD=0$ and $cD=0$.
Thus, $D$ is a zero divisor. If there is a bialgebra  map
$f:A\to K$  with $K$ a  Hopf algebra, then $f(D)$ is invertible and
so $f(b)=0$.

 Also, being $A$ commutative, we have that  
$\J=\id$, and the left inverse of a matrix with commuting entries is the same as a right
inverse, so hypothesis of Theorem \ref{teomain-new} are fullfilled. Hence,
$H(c)$ is a Hopf algebra. In this concrete example,
we see clearly that $b$ and $c$ are killed when inverting $D$, and we get the isomorphism
\[
H(c)\cong \frac {A}{(b=c=0)}[(ad-bc)^{-1}]
=\Bbbk[a,d][(ad)^{-1}]=\Bbbk[a^{\pm 1},d^{\pm 1}],
\]
with $\Delta (a)=a\ot a$ and $\Delta (d)=d\ot d$, so 
$H(c)\cong \Bbbk[\Z\times \Z]$.

\subsubsection{Quantum linear spaces}
We end the paper with another example that the canonical map is not necessarily injective. From our calculations one 
obtains another
proof, under certain hypothesis on the braiding, 
of the very well-known fact that a Nichols algebra associated with a quantum linear space
of dimension $n$
is realizable as a braided Hopf algebra in the category of $\Bbbk[\Z^{n}]$-comodules.

Let $V$ be a finite-dimensional vector space with basis $\{x_{i}\}_{1\leq i\leq n}$ and consider the   
following \textit{diagonal} braiding on it:
\[
c(x_i\ot x_j)=q_{ij}x_j\ot x_i \qquad \text{ for all }1\leq i,j\leq n,
\]
 where $q_{ij}\in\Bbbk^\times $ satisfy that $q_{ij}q_{ji}=1$ for $i\neq j$ and 
 $q_{ii}$ are primitive 
$N_i$-roots of 1, with $N_{i} \in \N$ and $N_{i}>1$. Then the Nichols algebra $\B(V,c)$ has generators
$x_1,\dots,x_n$ and relations
\[
x_ix_j=q_{ij}x_jx_i,\hskip 1cm x_i^{N_ i}=0.
\]
In particular, $\dim\B(V,c)=\prod_{i=1}^n N_i$, and a volume element is given by 
$\b=x_1^{N_1-1}\cdots x_n^{N_n-1}$.
Notice that all $N_i$ may be different, and not necessarily equal to 2. That is,
this is a family of  non-quadratic, non-homogeneous algebras.

Recall that, for a braiding $c$ of diagonal type, the FRT relations of $A(c)$ are given by
\[
q_{k\ell}t_i^kt_ j^\ell=
q_{ij}t_j^\ell t_i^k \qquad \text{ for all }1\leq i,j\leq n.
\]
This implies in particular that $A(c)$ is non-commutative if $q_{k\ell}\neq q_{ij}$.
Besides, it holds that 
$q_{\ell k}t_j^\ell t_ i^k=
q_{ji}t_i^k t_j^\ell
$ and consequently
\[
t_i^kt_ j^\ell=
q_{k\ell}^{-1}q_{ij}t_j^\ell t_i^k=
q_{k\ell}^{-1}q_{ij}
q_{\ell k}^{-1}
q_{ji}t_i^k t_j^\ell.
\]
By our assumptions on the braiding, this is nothing else that $t_i^kt_ j^\ell=t_i^kt_ j^\ell$ 
if $i\neq j$ and $k\neq \ell$. On the other hand, if $i=j$, we get
\[
t_i^kt_ i^\ell=
q_{k\ell}^{-1}q_{ii}
q_{\ell k}^{-1}
q_{ii}t_i^k t_i^\ell.
\]
Thus, for $k\neq \ell$ we obtain that 
$t_i^kt_ i^\ell=
q_{ii}^2t_i^k t_i^\ell$. If moreover $N_{i}\neq 2$, it holds that  $t_i^kt_ i^\ell=0$. 
Similarly, we have 
that 
\begin{align*}
t_i^kt_ j^k & =
q_{ kk}^{-2}
t_i^k t_j^k
& \text{ for }i\neq j \text{ and }k=\ell, \\ 
(t_i^k)^2 & =
q_{kk}^{-2}q_{ii}^2 (t_i^k )^2 & \text{ 
 for }i=j \text{ and }k=\ell. 
\end{align*}
From the considerations above we get the following lemma: 
\begin{lem}\label{lem:conditions-diagonal} Under the assumptions above, for $i\neq j$ we have:
\begin{enumerate}
 \item[$(a)$] If $q_{jj}^{-2}q_{ii}^2\neq 1$ then 
$(t_i^j)^2=0$.
\item[$(b)$] If $N_k\neq 2 $, then $t_i^kt_ j^k= 0 =t_k^it_k^j$.
\end{enumerate}
\end{lem}

\begin{coro} Assume that for $i\neq j$ it holds that $q_{jj}^{-2}q_{ii}^2\neq 1$ and 
$N_k\neq 2 $ for all $1\leq k\leq n$. Then 
\begin{enumerate}
\item[$(a)$] The quantum determinant is $D=(t_{1}^{1})^{N_{1}-1}(t_{2}^{2})^{N_{2}-1}\cdots (t_{n}^{n})^{N_{n}-1}$.
\item[$(b)$] $t_{i}^{j}=0$ for all $i\neq j$, and $t_k^k$ is group-like for all $1\leq k \leq n$ as elements in $H(c)$.
\item[$(c)$] $H(c)\cong \Bbbk[(t_1^1)^{\pm1},(t_2^2)^{\pm1},\dots,(t_n^n)^{\pm1}]\simeq \Bbbk [\Z^{n}]$.
In particular,
the canonical map is not injective if $q_{k\ell}\neq q_{ij}$ for some $1\leq i,j,k,\ell\leq n$.
\end{enumerate}
\end{coro}

\begin{proof}
$(a)$ The claim follows by a direct computation. Indeed, by Lemma \ref{lem:conditions-diagonal},  for $2\leq \ell \leq N_{k}-1$ one has that  
\[\lambda (x_k^\ell)=
\sum_{i_{1},\ldots, i_{\ell}}(t_k^{i_{1}}t_k^{i_{2}}\cdots t_k^{i_{\ell}})\ot x_{i_{1}}x_{i_{2}}\cdots x_{i_{\ell}}=
(t_k^k)^{\ell}\ot x_k^\ell.
\]
In particular, this implies that $(t_k^k)^{\ell}$ is a group-like element for all $1\leq k\leq n$ and $2\leq \ell \leq N_{k}-1$.
Hence, $\lambda(x_1^{N_1-1}\cdots x_n^{N_n-1}) = (t_{1}^{1})^{N_{1}-1}(t_{2}^{2})^{N_{2}-1}\ot x_1^{N_1-1}\cdots x_n^{N_n-1}$ and
the assertion is proved.

$(b)$ Since $q_{k\ell}t_i^kt_ j^\ell=
q_{ij}t_j^\ell t_i^k$,
it follows that
$q_{ij}t_i^it_ j^j= 
q_{ij}t_j^j t_i^i$, which implies that $t_i^it_ j^j=
t_j^j t_i^i$ for all $1\leq i,j\leq n$. Thus, for all $1\leq k \leq n$ we may write $D=(t_{k}^{k})^{N_{k}-1}D'$ for some $D'\in A(c)$. Since 
$D$ is invertible in $H(c)$ and $N_{k}>1$, we have that $t_{k}^{k}$ is a unit in $H(c)$. 
On the other hand, as $t_i^kt_j^k=0$ for $i\neq j$, it follows that
$t_i^kt_k^k=0$ for $i\neq k$, from which follows that $t_i^k=0$ for $i\neq k$.
The last claim follows from the very definition of the comultiplication in $H(c)$.

$(c)$ From the considerations above, we obtain
\[
H(c)=\frac{A}{(t_i^j:i\neq j)}[D^{-1}]
=\Bbbk[t_1^1,t_2^2,\dots,t_n^n][D^{-1}]=
\Bbbk[(t_1^1)^{\pm 1},(t_2^2)^{\pm 1},\dots,(t_n^n)^{\pm 1}]
\cong \Bbbk[\Z ^n].
\]
\end{proof}

\Addresses

\end{document}